\documentclass[final,onefignum,onetabnum]{siamart220329}

\usepackage{lipsum}
\usepackage{amsfonts}
\usepackage{graphicx}
\usepackage{epstopdf}
\ifpdf
\DeclareGraphicsExtensions{.eps,.pdf,.png,.jpg}
\else
\DeclareGraphicsExtensions{.eps}
\fi

\usepackage{mathtools}
\usepackage[table]{xcolor}
\usepackage{multirow}

\newcommand{\R}{\mathbb R}
\newcommand{\HO}{H}
\newcommand{\HP}{H_D}
\newcommand{\RL}{{\mathcal R}}
\newcommand{\NC}{{k_0}}%NC : neighbours count
\newcommand{\MC}{{k_1}}%NC : multiplicity count
\newcommand{\Gw}{G_i^{\bot_{B_i}}}
\newcommand{\bw}{b_{G_i^{\bot_{B_i}}}}
\newcommand{\Bw}{B_{G_i^{\bot_{B_i}}}}
\newcommand{\ccg}{\cellcolor{green!40}}
\newcommand{\cco}{\cellcolor{orange!50}}
\newcommand{\ccr}{\cellcolor{red!50}}

\newsiamremark{remark}{Remark}
\newsiamremark{hypothesis}{Hypothesis}
\crefname{hypothesis}{Hypothesis}{Hypotheses}
\newsiamthm{gevp}{GEVP}
\crefname{gevp}{GEVP}{GEVPs}

\headers{A robust preconditioner for $\mathbf{H}(\mathbf{\lowercase{curl}})$ problems}{N. Bootland, V. Dolean, F. Nataf, and P.-H. Tournier}

\title{A robust and adaptive GenEO-type domain decomposition preconditioner for $\mathbf{H}(\mathbf{\lowercase{curl}})$ problems in three-dimensional general topologies\thanks{Revised 16 May 2025.
		\funding{The first two authors gratefully acknowledge support from the EPSRC grant EP/S004017/1.}}}

\author{Niall~Bootland\thanks{STFC Rutherford Appleton Laboratory, Harwell Campus, Chilton, OX11 0QX, UK
		(\email{niall.bootland@stfc.ac.uk}).}
	\and Victorita~Dolean\thanks{Department of Mathematics and Statistics, University of Strathclyde, Livingstone Tower, 26 Richmond Street, Glasgow, G1 1XH, UK \emph{and} Laboratoire J.A.~Dieudonn\'e, CNRS, University C\^ote d'Azur, Parc Valrose, 28 Avenue Valrose, Nice, 06108, France
		(\email{work@victoritadolean.com}).}
	\and Fr\'ed\'eric~Nataf\thanks{Laboratoire J.L.~Lions, Sorbonne University, 4 Place Jussieu, Paris, 75005, and INRIA Team Alpines, France
		(\email{frederic.nataf@sorbonne-universite.fr}, \email{pierre-henri.tournier@sorbonne-universite.fr}).}
	\and Pierre-Henri~Tournier\footnotemark[4]}

\usepackage{amsopn}
\DeclareMathOperator{\diag}{diag}

\ifpdf
\hypersetup{
	pdftitle={A robust and adaptive GenEO-type domain decomposition preconditioner for H(curl) problems in three-dimensional general topologies},
	pdfauthor={N. Bootland, V. Dolean, F. Nataf, and P.-H. Tournier}
}
\fi

\begin{document}

\maketitle
% REQUIRED
\begin{abstract}
	In this paper we design, analyse and test domain decomposition methods for linear systems of equations arising from conforming finite element discretisations of positive Maxwell-type equations, namely for $\mathbf{H}(\mathbf{curl})$ problems. It is well known that convergence of domain decomposition methods rely heavily on the efficiency of the coarse space used in the second level. We design adaptive coarse spaces that complement a near-kernel space made from the gradient of scalar functions. The new class of preconditioner is inspired by the idea of subspace decomposition, but based on spectral coarse spaces, and is specially designed for curl-conforming discretisations of Maxwell's equations in heterogeneous media on general domains which may have holes. We also address the practical robustness of various solvers in the case of non-trivial topologies and/or high aspect ratio of the domain.
\end{abstract}

% REQUIRED

% REQUIRED

\section{Motivation}
\label{sec:motivation}

In this work we focus on the efficient solution of the following Maxwell problem:
\begin{equation}
	\label{eq:positivemaxwell}
	\begin{aligned}
		\nabla \times (\mu^{-1}\nabla\times \boldsymbol{E} ) + \gamma \varepsilon \boldsymbol{E} &= \boldsymbol{f} & &\text{in } \Omega,\\
		\boldsymbol{E} \times \boldsymbol{n} &= 0 & &\text{on } \partial\Omega.
	\end{aligned}
\end{equation}
Here $\boldsymbol{E}$ is the vector-valued electric field, $\boldsymbol{f}$ is a source term, while coefficients $\mu$ and $\varepsilon$ are electromagnetic parameters which are uniformly bounded and strictly positive but which we allow to be heterogeneous. Further, $\gamma>0$ is a positive parameter which is allowed to be very small. Let $\Omega$ denote the computational domain (possibly non-convex and not simply connected) and $\boldsymbol{n}$ denote the outward normal to $\partial\Omega$. For ease of presentation we consider the Dirichlet case in \eqref{eq:positivemaxwell} but analogous results apply for other boundary conditions, as examples in our numerical results will show. Upon appropriate discretisation we must solve a linear system $A \mathbf{E} = \mathbf{f}$. This problem, although positive definite, remains challenging from the solution methods point of view. To date, the reference approach for tackling this problem is the celebrated algorithm of Hiptmair--Xu from \cite{Hiptmair:2007:NSP}, which was identified by the U.S. Department of Energy in a 2008 report as one of the top ten recent breakthroughs in computational science.\footnote{\href{https://science.osti.gov/-/media/ascr/pdf/program-documents/docs/Breakthroughs_2008.pdf}{Report of The Panel on Recent Significant Advancements in Computational Science.}} This preconditioner is based upon a splitting of the space (giving a Helmholtz decomposition), here $\mathbf{H}_0(\mathbf{curl},\Omega)$ for $\Omega$ convex, by isolating the kernel of the curl operator
\begin{align*}
	\mathbf{H}_0(\mathbf{curl},\Omega) = (H_0^1(\Omega))^3 + \mathop\mathbf{grad} H_0^1(\Omega).
\end{align*}
The auxiliary space $H_0^1(\Omega)$ then uses a nodal (i.e., Lagrangian) discretisation. The matrix version of the preconditioner is given as
\begin{align*}
	\mathcal{M}^{-1}_\mathrm{AMS} = \diag(A)^{-1} + P ( \widetilde{L} + \gamma \widetilde{Q})^{-1} P^T + \gamma^{-1} C L ^{-1} C^T,
\end{align*}
where $\widetilde{L}$ is the nodal discretisation of the $\mu^{-1}$-weighted vector Laplacian operator, $\widetilde{Q}$ is the nodal $\varepsilon$-weighted vector mass matrix, $P$ is the matrix form of the nodal interpolation operator between the N\'ed\'elec space and nodal element space, $L$ is the nodal $\varepsilon$-weighted scalar Laplacian, and $C$ is the ``gradient matrix'', which is exactly the null space matrix of the discretisation of the curl--curl operator. This approach is known by several different names, the Hiptmair--Xu (HX) preconditioner, the (nodal) auxiliary space preconditioner (ASP), and the auxiliary-space Maxwell solver (AMS); here we will use the abbreviation AMS, as in \cite{kolev2009parallel} where a parallel implementation is considered and tested along with similar variants. The robustness of the AMS preconditioner with respect to the mesh size, heterogeneities in the coefficients and the topology of the domain has been thoroughly studied from both theoretical and practical points of view; see, e.g., \cite{hiptmair2006auxiliary,Hiptmair:2007:NSP,Hiptmair:2020:review,Xu:OMM:2009,Xu:RBH:2011,kolev2009parallel,hiptmair2019discrete,ayuso2017auxiliary,hu2021convergence,hu2023convergence}. Nonetheless, to the best of our knowledge there is no theoretical guarantee that the AMS preconditioner is efficient for elongated domains with holes, such as the one exhibited in Fig.~\ref{fig:holes} (see \Cref{sec:numerical_results}). Our numerical results in \Cref{sec:numerical_results} show that, indeed, the iteration counts of AMS increase in such cases; see \Cref{tab:beam_holes_both,tab:gamma_holes_both}.

A natural question is then whether we can propose a preconditioner that would remedy this issue. We choose here to consider this question from the point of view of domain decomposition methods, especially in order to address non-trivial topologies. Overlapping Schwarz methods were considered in \cite{toselli2000overlapping,pasciak2002overlapping}, for which a best condition number bound was very recently proved in \cite{liang:2023:sharp}. Non-overlapping and substructuring methods are considered in \cite{hu2003nonoverlapping,hu2004substructuring,toselli2006dual,dohrmann2012iterative}, with some analysis in the case of jumps in the physical parameters provided in \cite{hu2013discrete}, although under somewhat restrictive assumptions. Further, a non-overlapping BDDC algorithm is constructed in \cite{dohrmann2016bddc} and gives improved bounds in relation to the ratio between coarse and fine meshes but still with dependence on the physical parameters akin to \cite{toselli2006dual}; see also results in \cite{zampini2018balancing}. Within overlapping domain decomposition methods, in order to obtain a robust approach a second level is typically required. A good way to provide appropriate information is via spectral coarse spaces, exemplified by the GenEO approach \cite{spillane2014abstract,Dolean:2015:IDDSiam}, which use local generalised eigenvalue problems (GEVPs) and can provide robustness for highly heterogeneous problems. The approach has also seen theoretical extensions to indefinite and nonself-adjoint problem \cite{bootland2022geneo,bootland2023overlapping} as well as to saddle point systems \cite{nataf2023geneo}. It has also been adapted for wave propagation in the form of high-frequency heterogeneous Helmholtz problems \cite{bootland2021comparison,bootland2022can}; see also the more general overview for wave propagation problems in \cite{bootland2023several}. We additionally mention an alternative approach called MS-GFEM \cite{babuska2011optimal,ma2022novel,ma2022error}, which uses optimal local approximations spaces through use of local eigenproblems defined on generalised harmonic spaces and has also been applied to the high-frequency heterogeneous Helmholtz problem \cite{ma2023wavenumber}. 

Given the success of spectral coarse spaces, such an approach is natural to consider here in the case of the positive Maxwell equation. Note that, since the operator is symmetric positive definite, the GenEO theory developed in \cite{spillane2014abstract} applies directly, but in practice a difficulty arises in solving the local eigenproblems required to build the coarse space. Indeed, the purpose of the local GEVPs is to capture the local components of the near-kernel of the global problem. In our present case of the Maxwell problem \eqref{eq:positivemaxwell} and for small $\gamma$, the gradients of $H^1$ functions contribute to the sought coarse space since their curl is zero. There are many such functions and an iterative eigenvalue solver has difficulties finding all of them. Thus, we \emph{a priori} include the gradients of $H^1$ functions in the coarse space, but even this coarse space is not always sufficient. For instance, when the domain has a non-trivial topology (e.g., holes in the domain), which features naturally in various practical applications, the kernel of the curl operator is larger than the gradient of $H^1$ functions. It can be computed but requires some effort; see \cite{Pellikka:2013:HCC}. Moreover, even when the complete kernel of the curl operator is known, it may not be sufficient to build a good coarse space, such as when there are jumps in the coefficients.

With this in mind, our proposed solution is to start from some approximate knowledge of appropriate coarse space vectors, namely through the gradients of $H^1$ functions, and enrich this space by locally solving GEVPs in each subdomain. This ensures the size of the enrichment is small and so computable by the eigensolver. Our approach yields an adaptive coarse space in the GenEO spirit, but with modification required to account for the large near-kernel of the global problem. Overall, our main contributions in this work are:
\begin{enumerate}
	\item A general framework for the extension of the GenEO domain decomposition method to cases where the analytic near-kernel of the problem is large.
	\item An application of the framework to the special case of the Maxwell system.
	\item Extensive numerical tests and comparisons with the AMS preconditioner which show robustness with respect to:
	\begin{itemize}
		\item the geometry and its topology;
		\item how close the problem is to being singular due to a small zeroth-order term given by the value of $\gamma$. 
	\end{itemize}
\end{enumerate}
Our method ensures robustness across a wide range of challenging scenarios in terms of iteration counts, as seen in the numerical results of \Cref{sec:numerical_results}, and to our knowledge this is the first method to be provably robust with respect to topology of the domain and heterogeneities in the coefficients; this comes at a price of a coarse space which is large since it has to contain all gradients of scalar functions, plus a small contribution of local eigenvectors coming from a spectral GenEO coarse space. Nevertheless, the coarse space size is typically five times smaller than the original problem and the gradient part of the coarse matrix retains the sparsity pattern of the underlying finite element discretization. Further investigation is needed in order to provide an efficient way of solving the coarse problem.

Note that the preconditioner we propose can also be applied to time dependent Maxwell's equations, with the parameter $\gamma = (\Delta t)^{-1}$ in this case, where $\Delta t$ is the time discretisation step. Additionally, the positive Maxwell operator is used in \cite{Bachinger:2006:efficient} to build preconditioners in the context of time-periodic eddy current problems. Another class of related problems is the time-harmonic Maxwell system where the parameter $\gamma = -\omega^2$ can be negative and $\varepsilon$ complex valued. A different set of complications arise depending on whether we have high-frequency or low-frequency problems, as determined by the value of the angular frequency $\omega$. For this reason, the construction of fast iterative solvers for time-harmonic Maxwell systems is a problem of great current interest. Typically in the high-frequency case, the solution is very oscillatory and the size of the discretised problem increases with the frequency. The nature of the equations becomes very different from an elliptic problem and the method we propose is not suitable. The design of scalable coarse spaces that remain effective under strong indefiniteness is a very active field of research; see, e.g., \cite{ma2023wavenumber,bootland2023several} and references therein. Also, in this case, a domain decomposition preconditioner can be built from the absorptive case, relying on an idea originating in \cite{Erlangga:2004:OCP} but see also \cite{magolu2001incomplete}; this is often called ``shifted Laplacian preconditioning''. In \cite{Graham:2017:DDP}, the authors propose a two-level domain decomposition preconditioner for the Helmholtz equation based on adding absorption. These ideas are extended to Maxwell's equation in \cite{Bonazzoli:2019:DDP} through a combination of techniques inspired from the indefinite case but also from the positive definite one.

In the low-frequency regime, i.e., when the angular frequency $\omega \ll 1$, the problem is close to being singular and to deal with this problem, one option utilised in \cite{greif2007preconditioners,li2012parallel} is to consider a mixed form of these equations, which ensures stability and well-posedness, by introducing an artificial variable (Lagrange multiplier) $p$ to obtain the system
\begin{subequations}
	\label{eq:MaxwellExtended}
	\begin{align}
		\nabla \times (\mu^{-1} \nabla \times \boldsymbol{E} ) -{\omega^2} \varepsilon \boldsymbol{E} + \varepsilon \nabla p &= \boldsymbol{f} & &\text{in } \Omega,\\
		\nabla \cdot (\varepsilon \boldsymbol{E}) &= 0 & &\text{in } \Omega,\\
		\boldsymbol{E} \times \boldsymbol{n} &= \mathbf{0} & &\text{on } \partial\Omega,\\
		p &= 0 & &\text{on } \partial\Omega.
	\end{align}
\end{subequations}
If we discretise the system \eqref{eq:MaxwellExtended} using N\'ed\'elec elements for $\boldsymbol{E}$ and Lagrange nodal elements for $p$ we obtain the following saddle point system
\begin{align}
	\label{eq:MaxwellSaddlePoint}
	\mathcal{K} \left(\begin{array}{c} \mathbf{E} \\ \mathbf{p}\end{array}\right) \vcentcolon= \left(\begin{array}{cc}
		{K} - {\omega^2} M & B^T \\
		B & 0
	\end{array}\right) \left(\begin{array}{c} \mathbf{E} \\ \mathbf{p}\end{array}\right) = \left(\begin{array}{c}\mathbf{f} \\ \mathbf{0}\end{array}\right),
\end{align}
where $K \in \mathbb{R}^{n \times n}$ corresponds to the discretisation of $\nabla \times (\mu^{-1}\nabla\times)$ and has a large kernel of dimension $m$, $M \in \mathbb{R}^{n \times n}$ is the $\varepsilon$-weighted mass matrix computed on the edge-element space, and $B \in \mathbb{R}^{m \times n}$ is a full row rank matrix corresponding to the divergence constraint. As a consequence, the stiffness matrix ${\mathcal{K}}$ is symmetric but indefinite.

A block preconditioner is proposed in \cite{greif2007preconditioners}, taking the form
\begin{align}
	\label{eq:MaxwellBlockPrec}
	\mathcal{P}_{M,L} \vcentcolon= \left(\begin{array}{cc}
		\mathcal{P}_{M} & 0 \\
		0 & {L}
	\end{array}\right), \quad \mathcal{P}_{M} \vcentcolon= K + \gamma M, \quad \gamma = 1 -\omega^2 > 0,
\end{align}
where {$L$} is the nodal {$\varepsilon$-weighted Laplace operator}; the challenge now being solving efficiently with $\mathcal{P}_{M}$. Here $\mathcal{P}_{M}$ represents the discretisation of the \emph{positive Maxwell} problem \eqref{eq:positivemaxwell} and thus will be amenable to our methodology. We note that in the limit case $\omega = 0$, the leading block in \eqref{eq:MaxwellSaddlePoint} is maximally rank deficient and the approach of \cite{estrin2015nonsingular} applies. This further inspired related work that extends to the case of block preconditioners for $\omega$ small \cite{liang2021preconditioners}. We also mention earlier work on mixed approximation of magnetostatic problems in \cite{perugia1999linear}. More recent literature on solving general saddle point systems can be found in, e.g., \cite{pestana2015natural,notay2019convergence}.

Since a key part of many Maxwell-type solvers is the solution of a positive Maxwell problem \eqref{eq:positivemaxwell}, as in \eqref{eq:MaxwellBlockPrec} and, for example, \cite{cessenat1996mathematical,ciarlet1999fully}, in what follows we will focus on the efficient solution to the discretisation of this system, given by
\begin{align}
	\label{eq:discretemaxwell}
	A\mathbf{U} \vcentcolon= (K+\gamma M)\mathbf{U} = \mathbf{f},
\end{align}
where we now use $\mathbf{U}$ as our unknown. We reiterate that the matrix $K$ has a large kernel (since all gradients of $H^1$ functions are part of the kernel of the curl operator) and so designing efficient preconditioners for this problem can be challenging when $\gamma$ is small, especially in a general case with a non-trivial topology.

The remainder of this work is laid out as follows. In \Cref{sec:domain_decomposition_preliminaries} we introduce notation and preliminaries required for our theory. We then prove condition number bounds for our newly developed two-level additive Schwarz preconditioner in \Cref{sec:ASwithFSL}. A selection of illustrative numerical results are then presented in \Cref{sec:numerical_results}. These detail examples with heterogeneity and domains with holes and show the robustness of our approach in all cases, even for elongated domains with holes for which the iteration counts of AMS increase with the difficulty of the test case. Finally, we summarise our conclusions in \Cref{sec:conclusions}.

\section{Notation and preliminaries}
\label{sec:domain_decomposition_preliminaries}

In the work that follows we will use the idea of subspace decomposition from the auxiliary space preconditioner to build a new domain decomposition preconditioner. The latter will contain a second level with two types of components: one arising from the part of the kernel of matrix $K$ that is readily available (i.e., related to the gradient of $H^1$ functions in our case), and the other stemming from local spectral information, which will help tackle the heterogeneous nature of the problem and the potential for holes in the domain.

Both cases can be treated from a theoretical point of view in the general framework of the fictitious space lemma (FSL) of Nepomnyaschikh \cite{Nepom:1991:DDM}, which can be considered the Lax--Milgram theorem of domain decomposition. For the sake of completeness, we briefly introduce the problem formulation and notation used (which is similar to that in \cite{nataf2020mathematical}), along with some theoretical results that will be useful in \Cref{sec:ASwithFSL}. We then detail the FSL, presented in the form originating in Griebel and Oswald \cite{griebel1995abstract}, before applying it within our context.

We suppose that our problem is posed on an arbitrary finite domain $\Omega \subset \mathbb{R}^3$ and is given in weak form. Namely, for an appropriate Hilbert space $H$, of functions on $\Omega$, we wish to find $u \in H$ such that
\begin{align*}
	a(u,v) = l(v), \qquad \text{for all } v \in H,
\end{align*}
where $a(\cdot,\cdot)$ is a symmetric coercive bilinear form on $H$ and $l(\cdot)$ is a linear functional on $H$. In our case of the positive Maxwell problem \eqref{eq:positivemaxwell}, highlighting the functions as vectorial by using bold, our bilinear form is
\begin{align}
	\label{eq:bilinearformposmax}
	a(\mathbf{u},\mathbf{v}) = \int_{\Omega} \mu^{-1} (\nabla \times \mathbf{u}) \cdot (\nabla \times \mathbf{v}) + \gamma \varepsilon \mathbf{u} \cdot \mathbf{v} \, \mathrm{d}\mathbf{x},
\end{align}
while the linear functional is
\begin{align*}
	l(\mathbf{v}) = \int_{\Omega} \mathbf{f} \cdot \mathbf{v} \, \mathrm{d}\mathbf{x}.
\end{align*}
We assume the problem is discretised on a mesh $\mathcal{T}_{h}$ by an appropriate finite element method, say using N\'ed\'elec elements, to obtain the linear system \eqref{eq:b}, with symmetric positive definite system matrix $A$.

Consider an overlapping domain decomposition $\Omega = \cup_{1 \le i \le N}\Omega_i$ where the domain $\Omega$ is decomposed into $N$ subdomains, which we assume to conform to the mesh $\mathcal{T}_{h}$. We will later need the use of local Neumann matrices on each subdomain $\Omega_i$. These are given by restricting the bilinear form \eqref{eq:bilinearformposmax} to be only over $\Omega_i$ without enforcing any new boundary conditions on the interior of $\Omega$, that is, we define $A_{i}^\mathrm{Neu}$ to be the equivalent discretisation of
\begin{align*}
	a_{\Omega_i}(\mathbf{u},\mathbf{v}) = \int_{\Omega_i} \mu^{-1} (\nabla \times \mathbf{u}) \cdot (\nabla \times \mathbf{v}) + \gamma \varepsilon \mathbf{u} \cdot \mathbf{v} \, \mathrm{d}\mathbf{x},
\end{align*}
where the boundary conditions from the underlying problem, i.e., from \eqref{eq:positivemaxwell}, are imposed on $\partial\Omega_i \cap \partial\Omega$. Note that, since $\gamma > 0$, the Neumann matrices are SPD; this will ensure a necessary eigenproblem (to be defined in \cref{gevp:tauthresholdAS}) is always well-posed.

Let the degrees of freedom (dofs) for the discretised problem \eqref{eq:discretemaxwell} be ordered into a set $\mathcal{N}$ of size $n$. For each subdomain $\Omega_i$ we suppose the local dofs contained within the subdomain are given by $\mathcal{N}_{i} \subset \mathcal{N}$, with size $n_i$. We define the restriction operators from the set $\mathcal{N}$ to $\mathcal{N}_{i}$, respectively for each $i$, as $R_{i} \colon \mathbb{R}^{n} \rightarrow \mathbb{R}^{n_i}$. The corresponding matrices $R_{i}^{T}$, which constitute extension by zero outside of $\mathcal{N}_{i}$, are called prolongation (or extension) operators. We also require a partition of unity to glue together solutions on each subdomain $\Omega_i$. At the algebraic level this is given by non-negative diagonal matrices $D_{i} \in \mathbb{R}^{n_i \times n_i}$ such that
\begin{align}
	\label{eq:PoU}
	\sum_{i=1}^{N} R_{i}^{T} D_{i} R_{i} = I,
\end{align}
where $I \in \mathbb{R}^{n \times n}$ is the identity. We further introduce several constants related to the domain decomposition which will be useful later on. Firstly, we define the maximum multiplicity of the interaction between subdomains plus one as
\begin{align}
	\label{eq:k0}
	k_{0} \vcentcolon= \max_{1 \le i \le N} \# \left\lbrace j \, \vert \, R_{j} A R_{i}^{T} \neq 0 \right\rbrace,
\end{align}
where $\#$ gives the size of a set. Note that this is bounded from below by the number of colours required to colour a decomposition so that no two neighbouring subdomains share the same colour and is typically a small constant for standard decomposition techniques. Secondly, we define the maximum multiplicity of subdomain intersection, $k_{1}$, as the largest number of subdomains whose intersection has non-zero measure. We now present two auxiliary lemmas from \cite{Dolean:2015:IDDSiam}, related to these constants, that will be used later in our theory in \Cref{sec:ASwithFSL}.

\begin{lemma}[Lemma~7.9 in \cite{Dolean:2015:IDDSiam}]\label{Lemma_k0}
	Suppose that we have a set of rectangular matrices $Q_i \in \mathbb{R}^{n \times n_i}$, for $1 \le i \le N$, and that $A \in \mathbb{R}^{n \times n}$ is symmetric positive definite. Define the constant
	\begin{align*}
		\widetilde{k_{0}} \vcentcolon= \max_{1 \le i \le N} \# \left\lbrace j \, \vert \, Q_{j} A Q_{i}^{T} \neq 0 \right\rbrace.
	\end{align*}
	Consider the vectors $\mathbf{U}_i \in \mathbb{R}^{n_i}$ for $1 \le i \le N$, then we have the following estimate
	\begin{align*}
		\left(\sum_{i=1}^{N} Q_i \mathbf{U}_i \right)^{T} A \left(\sum_{i=1}^{N} Q_i \mathbf{U}_i \right) \le \widetilde{k_{0}} \sum_{i=1}^{N} \mathbf{U}_{i}^{T} (Q_{i}^{T} A Q_{i}) \mathbf{U}_{i}.
	\end{align*}
\end{lemma}

\begin{lemma}[Lemma~7.13 in~\cite{Dolean:2015:IDDSiam}]\label{Lemma_k1}
	Suppose that the matrix $A \in \mathbb{R}^{n \times n}$ stems from the symmetric coercive bilinear form $a(\cdot,\cdot)$ on a domain $\Omega$ discretised on a mesh $\mathcal{T}_{h}$. Further, assume a domain decomposition as above, which yield equivalent local Neumann matrices $A_{i}^\mathrm{Neu}$, $1 \le i \le N$, and let $k_{1}$ be defined, as previously, as the maximum multiplicity of subdomain intersection. Consider a vector $\mathbf{U} \in \mathbb{R}^{n}$, then we have the bound
	\begin{align*}
		\sum_{i=1}^{N} (R_{i} \mathbf{U})^{T} A_{i}^\mathrm{Neu} (R_{i} \mathbf{U}) \le k_{1} \mathbf{U}^{T}A\mathbf{U}.
	\end{align*}
\end{lemma}

Our theory will build on the additive Schwarz domain decomposition approach. When used as a preconditioner it takes the form
\begin{align}
	\label{eq:AS}
	M_\mathrm{AS}^{-1} = \sum_{i=1}^{N} R_{i}^{T} A_{i}^{-1} R_{i},
\end{align}
where $A_{i} = R_{i} A R_{i}^{T}$ is the (local) Dirichlet matrix on $\Omega_{i}$, namely where Dirichlet conditions are implicitly assumed one extra layer of mesh elements away from $\partial\Omega_{i}\setminus\partial\Omega$. Since such a one-level approach is typically not scalable (though see \cite{bootland2022analysis}) we must enhance convergence through use of a second-level correction, i.e., a coarse space. This derives from a collection of column vectors $Z$, which should form a basis for a space $V_{0}$, called the coarse space. This further yields a coarse space operator $E = Z^{T} A Z$ along with the coarse correction operator $Z E^{-1} Z^{T}$. The second level can be incorporated in a number of ways; we consider one that maintains the symmetry of \eqref{eq:AS} given by
\begin{align}
	\label{eq:AS2}
	M_\mathrm{AS,2}^{-1} = Z E^{-1} Z^{T} + (I - P_{0}) M_\mathrm{AS}^{-1} (I - P_{0}^{T}),
\end{align}
where $P_{0} = Z E^{-1} Z^{T} A$ is the $A$-orthogonal projection on the coarse space $V_0$. Note that vectors in the coarse space are typically denoted with a subscript zero, which relates to the coarse space as opposed to a subdomain $\Omega_i$, $i \ge 1$. We suppose that the size of the coarse space is $n_{0}$ and a generic vector in $V_{0}$ is written as $\mathbf{U}_{0} \in \mathbb{R}^{n_0}$ in a similar way that $\mathbf{U}_{i} \in \mathbb{R}^{n_i}$ refers to a local vector on the subdomain $\Omega_i$, $1 \le i \le N$.

We now turn to the main theoretical result that we employ, the fictitious space lemma (FSL), before detailing how this fits within our current context.

\begin{lemma}[Fictitious space lemma, as presented by Griebel and Oswald \cite{griebel1995abstract}]\label{FSL}
	Consider two Hilbert spaces $H$ and $H_D$ with the respective inner products $(\cdot,\cdot)$ and $(\cdot,\cdot)_{D}$. Further, suppose that we have symmetric coercive bilinear forms $a \colon H \times H \rightarrow \R$ and $b \colon H_D \times H_D \rightarrow \R$. The operators $A$ and $B$ are defined as follows:
	\begin{itemize}
		\item Let $A \colon H \rightarrow H$ be such that $(Au,v)=a(u,v)$ for all $u,v\in H$;
		\item Let $B \colon H_D \rightarrow H_D$ be such that $(Bu_D,v_D)_D=b(u_D,v_D)$ for all $u_D,v_D\in H_D$.
	\end{itemize}
	Suppose that we have a linear surjective operator $\RL \colon H_D \rightarrow H$ verifying the following properties.
	\begin{itemize}
		\item Continuity: $\exists c_R > 0$ such that $\forall u_D\in H_D$ we have
		\begin{align*}
			a(\RL u_D,\RL u_D) \le c_R b(u_D,u_D).
		\end{align*}
		\item Stable decomposition: $\exists c_T>0$ such that $\forall u\in H$ $\exists u_D\in H_D$ with $\RL u_D=u$ and
		\begin{align*}
			c_T b(u_D,u_D) \le a(\RL u_D,\RL u_D)=a(u,u).
		\end{align*}
	\end{itemize}
	Consider the adjoint operator $\RL^* \colon H\rightarrow H_D$ given by
	$(\RL u_D, u) = (u_D, \RL^* u)_D$ for all $u_D\in H_D$ and
	$u\in H$. Then for all $u\in H$ we have the spectral estimate
	\begin{align*}
		c_T a(u,u) \le a\left(\RL B^{-1} \RL^* A u,\,u\right) \le c_R a(u,u).
	\end{align*}
	Thus, the eigenvalues of the preconditioned operator $\RL B^{-1} \RL^* A$ are bounded from below by $c_T$ and from above by $c_R$ and we have the spectral condition number estimate
	\begin{align*}
		\kappa_{2}({\RL B^{-1} \RL^* A}) \le {c_T^{-1}} {c_R}.
	\end{align*}
\end{lemma}
\begin{remark}
	In this lemma we have two Hilbert spaces (that are linked by the surjective operator $\RL$) and two symmetric coercive bilinear forms. The first one comes from our problem while the second is for the preconditioner. Under the assumptions of continuity and stable decomposition, the spectral estimate tells us that the spectral condition number of the preconditioned problem is bounded solely in terms of the constants $c_R$ and $c_T$. The matrix $A \in \R^{n \times n}$ is typically a symmetric positive definite matrix, stemming from the symmetric coercive bilinear form $a(\cdot,\cdot)$ on $\Omega$ corresponding to the problem to solve. Our main assumption is that the matrix $A$, defining the linear system we wish to solve, has a large near-kernel which can cause problems numerically. In other words, $A$ is a small perturbation from $K$, a matrix with a large kernel. Thus, we would like to deal with the near-kernel directly in the coarse space in order that we can primarily work in the orthogonal complement, avoiding the ill-conditioning in $A$ from the near-kernel. The near-kernel in the case of the positive Maxwell problem \eqref{eq:positivemaxwell} on a trivial topology domain $\Omega$ is given by the gradient of all $H^1(\Omega)$ functions, namely the kernel of the curl--curl operator.
\end{remark}

\begin{definition}[Split near-kernel coarse space]\label{def:near-kernelCS}
	Let $G$ be the easily computable part of the near-kernel of $A$; namely, in our case, vectors stemming from gradients of $H^1$ functions but not those related to any holes in the domain. We define the split near-kernel (SNK) coarse space $V_G \subset \R^{n}$ as the vector space spanned by the sequence $(R_i^T D_i G_i)_{1\le i \le N}$, where $G_i := R_i G$, so that $G \mathop \subset V_G$. Here we have split $G$ into contributions on subdomains and note that $\dim V_{G} > \dim G$ in general; we will see that such a splitting is necessary in numerical results. The coarse space matrix $Z \in \R^{n \times n_{0}}$ is a rectangular matrix whose columns are a basis of $V_0:=V_G$. The coarse space matrix is then defined in the usual way by $E = Z^TAZ$.
\end{definition}

Note that the space $V_G$ is the equivalent of the Nicolaides coarse space \cite{nicolaides1987deflation} for the Laplace problem or the rigid body modes coarse space for the linear elasticity problem; see \cite{Dolean:2015:IDDSiam}.

We now need to define all the other ingredients in the FSL (\Cref{FSL}). Our (discrete) problem space is $H = \mathbb{R}^{n}$ while the second Hilbert space $H_D$, for the preconditioning, is the product space of vectors stemming from the $n_{i}$ degrees of freedom on the local subdomains $\Omega_i$ along with the $n_0$ coarse space vectors
\begin{align}
	\label{eq:H_D}
	H_D \vcentcolon= \R^{n_0} \times \prod_{i=1}^N \R^{n_i}.
\end{align}
The bilinear form $b(\cdot,\cdot)$ is chosen to be given by the sum of local bilinear forms $b_i(\mathbf{U}_i, \mathbf{V}_i) \vcentcolon= (A_{i} \mathbf{U}_{i}, \mathbf{V}_{i})$ and a coarse space contribution
\begin{align}
	\label{eq:b}
	b(\mathcal{U},\mathcal{V}) &\vcentcolon= (E\mathbf{U}_0,\mathbf{V}_0) + \sum_{i=1}^{N} b_i(\mathbf{U}_i, \mathbf{V}_i), & \text{for all } \mathcal{U}, \mathcal{V} &\in H_D,
\end{align}
where the product space vectors are $\mathcal{U} = (\mathbf{U}_0, (\mathbf{U}_i)_{1 \le i \le N})$ and $\mathcal{V} = (\mathbf{V}_0, (\mathbf{V}_i)_{1 \le i \le N})$. Finally, the surjective operator $\RL_\mathrm{AS} \colon H_D \rightarrow H$ corresponding the additive Schwarz method is given by
\begin{align}
	\label{eq:RAS}
	\RL_\mathrm{AS}(\mathcal{U}) \vcentcolon= Z\mathbf{U}_0 + (I-P_0) \sum_{i=1}^{N} R_i^T \mathbf{U}_i,
\end{align}
where $P_{0} = Z E^{-1} Z^{T} A$ is the $A$-orthogonal projection on the coarse space $V_0$.

\begin{definition}[Orthogonal projections and orthogonal complements]\label{def:OrthPandCs}
	Let the matrices $B_i \in \R^{n_i \times n_i}$ correspond to the local symmetric coercive bilinear forms $b_i(\cdot,\cdot)$, for $i \ge 1$. Within $\R^{n_i}$ we define the $b_i$-orthogonal complement of $G_i$
	\begin{align}
		\label{eq:Giortho}
		\Gw \vcentcolon= \left\{ \mathbf{U}_i \in \R^{n_i} \ \vert \ \forall \, \mathbf{V}_i \in G_i, \ b_i(\mathbf{U}_i, \mathbf{V}_i ) = 0 \right\}.
	\end{align}
	As such, let $\bw$ denote the restriction of $b_i$ to $\Gw$. 
\end{definition}
We have a representation formula $\bw$:
\begin{lemma}
	The Riesz representation theorem gives the existence of a unique isomorphism $\Bw \colon \Gw \rightarrow \Gw$ into itself so that
	\begin{align*}
		(\Bw\,\mathbf{U}_i,\mathbf{V}_i) \vcentcolon= \bw(\mathbf{U}_i,\mathbf{V}_i) ,\,\forall \, \mathbf{U}_i, \mathbf{V}_i \in \Gw.
	\end{align*}
	Further, we let $\xi_{0i}$ be the $b_i$-orthogonal projection from $\R^{n_i}$ on $G_i$ parallel to $\Gw$. This also enables us to express the inverse of $\Bw$, which we will denote by $B_i^\dag$, through the following formula
	\begin{align}
		\label{eq:inverseBw}
		B_i^\dag = (I - \xi_{0i}) B_i^{-1}.
	\end{align}
\end{lemma}
\begin{proof}
In order to check formula \eqref{eq:inverseBw}, we have to show that
\begin{align}
	\label{eq:inverseBwShow}
	\Bw(I - \xi_{0i})B_i^{-1} y= y,\, \forall y \in \Gw.
\end{align}
Let $z \in \Gw$, using the fact that $I-\xi_{0i}$ is the $b_i$-orthogonal projection on $\Gw$, we have
\begin{align*}
	(\Bw(I - \xi_{0i}) B_i^{-1} y, z) = b_i((I - \xi_{0i}) B_i^{-1} y, z) = b_i(B_i^{-1} y, z) = (y,z).
\end{align*}
Since this equality holds for any $z \in \Gw$, this proves that \eqref{eq:inverseBwShow} holds and thus that $B_i^\dag$ in \eqref{eq:inverseBw} is the inverse of $\Bw$.
\end{proof}

\section{Additive Schwarz method using the FSL}
\label{sec:ASwithFSL}

In this section we will, in the first instance, apply the FSL to the two-level additive Schwarz algorithm which uses the split near-kernel (SNK) coarse space from \Cref{def:near-kernelCS}, the setup of which was detailed in the previous section. We then improve the condition number bound obtained to give a robust approach by enriching the underlying coarse space with appropriate generalised eigenfunctions from the orthogonal complements $\Gw$.

\subsection{A two-level additive Schwarz method with the split near-kernel coarse space}

In order to apply the FSL (\Cref{FSL}) we first need to check its assumptions. These results are obtained in the following two lemmas.

\begin{lemma}[Surjectivity and continuity of $\RL_\mathrm{AS}$]\label{lem:surj_cont}
	The operator $\RL_\mathrm{AS}$ defined in \eqref{eq:RAS} is surjective and continuous with continuity constant $c_R = k_0$, as defined in \eqref{eq:k0}.
\end{lemma}
\begin{proof}
	First, we note that for any $\mathbf{U} \in \HO$, due to the partition of unity property \eqref{eq:PoU}, we have
	\begin{align*}
		\mathbf{U} &= P_0 \mathbf{U} + (I - P_0) \mathbf{U} = P_0 \mathbf{U} + (I - P_0) \sum_{i=1}^N R_i^T D_i R_i \mathbf{U}.
	\end{align*}
	Now since $P_0 \mathbf{U} \in V_G$ there exists $\mathbf{U}_0$ such that $Z \mathbf{U}_0 = P_0 \mathbf{U}$. Therefore, we have
	\begin{align*}
		\mathbf{U} = \RL_\mathrm{AS}(\mathbf{U}_0, (D_i R_i \mathbf{U})_{1\leq i\leq N}),
	\end{align*}
	and hence $\RL_\mathrm{AS}$ is surjective.
	
	To prove continuity, we are required to estimate a constant $c_R$ such that for all $\mathcal{U} = (\mathbf{U}_0, (\mathbf{U}_i)_{1 \le i \le N}) \in \HP$ we have
	\begin{align*}
		a( \RL_\mathrm{AS} (\mathcal{U}), \RL_\mathrm{AS} (\mathcal{U})) &\le c_R \, b(\mathcal{U}, \mathcal{U}) = c_R \, \left[ (E\mathbf{U}_0, \mathbf{U}_0) + \sum_{i=1}^N (R_i A R_i^T \mathbf{U}_i, \mathbf{U}_i)\right].
	\end{align*}
	In this regard, using twice the $A$-orthogonality of $I-P_0$ along with \Cref{Lemma_k0}, we have the estimate
	\begin{align*}
		a(\RL_\mathrm{AS}(\mathcal{U}), & \RL_\mathrm{AS}(\mathcal{U}))\\ &= \left\| Z \mathbf{U}_0 + (I-P_0) \sum_{i=1}^N R_i^T \mathbf{U}_i \right\|_A^2 = \left\| Z \mathbf{U}_0 \right\|_A^2 + \left\| (I-P_0) \sum_{i=1}^N R_i^T \mathbf{U}_i \right\|_A^2\\
		&\le (E \mathbf{U}_0, \mathbf{U}_0) + \NC \sum_{i=1}^N \left\| R_i^T \mathbf{U}_i \right\|_A^2 
		\le \NC \left[ (E \mathbf{U}_0, \mathbf{U}_0) + \sum_{i=1}^N \left\| R_i^T \mathbf{U}_i \right\|_A^2 \right].
	\end{align*}
	Thus the estimate of the constant of continuity of $\RL_\mathrm{AS}$ can be chosen as $c_R \vcentcolon= \NC \ge 1$, as defined in \eqref{eq:k0}.
\end{proof}

\begin{lemma}[Stable decomposition with $\RL_\mathrm{AS}$]\label{lem:stable}
	The operator $\RL_\mathrm{AS}$ verifies the stable decomposition property, i.e., $\exists c_T>0$ such that $\forall \mathbf{U}\in H$ $\exists \mathcal{U}\in H_D$ with $\RL_\mathrm{AS}(\mathcal{U})= \mathbf{U}$ and
	\begin{align*}
		c_T b(\mathcal{U},\mathcal{U}) \le a(\RL_\mathrm{AS}(\mathcal{U}),\RL_\mathrm{AS}(\mathcal{U}))=a( \mathbf{U}, \mathbf{U}).
	\end{align*}
	In particular, such a constant $c_T$ is given by
	\begin{align}
		\label{eq:c_T}
		c_T \vcentcolon= \frac{1}{1 + \MC \tau_0},
	\end{align}
	where $\tau_0$ will be defined in \eqref{eq:tauOneLevelAS} and with $k_1$ as defined in \Cref{Lemma_k1}.
\end{lemma}
\begin{proof}
	Note that for $\mathbf{U} \in \HO$, we have
	\begin{align*}
		\mathbf{U} &= P_0 \mathbf{U} + (I - P_0) \mathbf{U} = P_0 \mathbf{U} + (I - P_0) \sum_{i=1}^N R_i^T D_i R_i \mathbf{U} \\
		&= P_0 \mathbf{U} + (I - P_0) \sum_{i=1}^N R_i^T D_i (I - \xi_{0i}) R_i \mathbf{U} + \underbrace{(I - P_0) \sum_{i=1}^N R_i^T D_i \xi_{0i} R_i \mathbf{U}}_{\mathop=0}.
	\end{align*}
	Let us expand on the last line and observe that, since $P_0 \mathbf{U} \in V_G$, there exists $\mathbf{U}_0$ such that $Z \mathbf{U}_0 = P_0 \mathbf{U}$ and, further, the third term is zero since $\sum_{i=1}^N R_i^T D_i \xi_{0i} R_i \mathbf{U} \in V_G$. Therefore, we have
	\begin{align*}
		\mathbf{U} = \RL_\mathrm{AS}(\mathcal{U}),\, \mathcal{U} = \left(\mathbf{U}_0, (D_i (I - \xi_{0i}) R_i \mathbf{U})_{1\leq i\leq N}\right).
	\end{align*}
	Determining the stability of this decomposition consists in the estimation of a constant $c_T > 0$ such that
	\begin{align*}
		c_T \left[(E\mathbf{U}_0, \mathbf{U}_0) + \sum_{j=1}^N (R_j A R_j^T D_j (I-\xi_{0j}) R_j \mathbf{U}, D_j (I-\xi_{0j}) R_j \mathbf{U})\right] \leq a(\mathbf{U}, \mathbf{U}).
	\end{align*}
	Introducing $A_j^\mathrm{Neu}$ and, in the second step, applying \Cref{Lemma_k1} we have
	\begin{align}
		\label{eq:estimateAS}
		\begin{split}
			\sum_{j=1}^N \left(R_j A R_j^T D_j (I-\xi_{0j}) R_j \mathbf{U}, D_j (I-\xi_{0j}) R_j \mathbf{U}\right) &\le \tau_0 \sum_{j=1}^N \left(A_j^\mathrm{Neu} R_j \mathbf{U}, R_j \mathbf{U}\right)\\ &\le \tau_0 \, \MC \, a(\mathbf{U}, \mathbf{U}),
		\end{split}
	\end{align}
	where we define
	\begin{align}
		\label{eq:tauOneLevelAS}
		\tau_0 \vcentcolon= \max_{1 \le j \le N} \max_{\mathbf{V}_j \in \R^{n_j}} \frac{\left(R_j A R_j^T D_j (I-\xi_{0j}) \mathbf{V}_j, D_j (I-\xi_{0j}) \mathbf{V}_j\right)}{\left(A_j^\mathrm{Neu} \mathbf{V}_j, \mathbf{V}_j\right)},
	\end{align}
	so that the first inequality holds. Now by applying \eqref{eq:estimateAS}, we obtain
	\begin{align*}
		b(\mathcal{U},\mathcal{U}) & = (E\mathbf{U}_0, \mathbf{U}_0) + \sum_{j=1}^N \left(R_j A R_j^T D_j (I-\xi_{0j}) R_j \mathbf{U}, D_j (I-\xi_{0j}) R_j \mathbf{U}\right) \\
		&\leq a(P_0\mathbf{U}, P_0\mathbf{U}) + \tau_0 \, \MC \, a(\mathbf{U}, \mathbf{U}) \le (1 + \MC \tau_0) \, a(\mathbf{U}, \mathbf{U}).
	\end{align*}
	Thus the stable decomposition estimate holds with the constant $c_T$ in \eqref{eq:c_T}.
\end{proof}

\begin{theorem}[Convergence of additive Schwarz with the split near-kernel coarse space]
	The spectral condition number estimate for the two-level additive Schwarz preconditioner $M_\mathrm{AS,SNK}^{-1} \vcentcolon= \RL_\mathrm{AS}B^{-1}\RL_\mathrm{AS}^*$ is
	\begin{align*}
		\kappa(M_\mathrm{AS,SNK}^{-1} A) \le (1 + \MC \tau_0) \NC.
	\end{align*}
\end{theorem}
The proof is a direct consequence of \Cref{lem:surj_cont,lem:stable}, which give the necessary criterion to apply the FSL (\Cref{FSL}).

We note that this spectral condition number depends on the constants $k_0$ and $k_1$ which are fixed and bounded for a given decomposition into subdomains. However, the remaining constant $\tau_0$ may be very large due to the shape and size of the domain or the heterogeneities in the coefficients of the problem, leading to a large condition number. The fix is to enlarge the coarse space by a generalised eigenvalue problem (GEVP) induced by the formula~\eqref{eq:tauOneLevelAS}. More precisely, we introduce in the following a generalised eigenvalue problem and the related GenEO coarse space.

\subsection{The GenEO coarse space}

We start by introducing the relevant GEVP that will yield the appropriate GenEO eigenvectors that should be added in addition to the split near-kernel coarse space.
\begin{gevp}[Generalised eigenvalue problem for the lower bound]\label{gevp:tauthresholdAS}
	For each of the subdomains $1 \le j \le N$, we introduce the generalised eigenvalue problem:
	\begin{align}
		\label{eq:eigAtildeBAS}
		\begin{array}{c}
			\text{Find } (\mathbf{V}_{jk},\lambda_{jk}) \in \R^{n_j} \setminus \{0\} \times \R
			\mbox{ such that}\\[1ex]
			(I-\xi_{0j}^T) D_j R_j A R_j^T D_j (I-\xi_{0j}) \mathbf{V}_{jk} = \lambda_{jk} A_j^\mathrm{Neu} \mathbf{V}_{jk}.
		\end{array}
	\end{align}
	Let $\tau>0$ be a user-defined threshold, we define $V_{j,\mathrm{GenEO}}^\tau \subset \R^{n}$ as the vector space spanned by the family of vectors $(R_j^T D_j (I - \xi_{0j}) \mathbf{V}_{jk})_{\lambda_{jk} > \tau}$ corresponding to eigenvalues larger than $\tau$. Further, let $V_\mathrm{GenEO}^\tau$ be the vector space spanned by the collection over all subdomains of the vector spaces $(V_{j,\mathrm{GenEO}}^\tau)_{1 \le j \le N}$.
\end{gevp}
\begin{remark}
 The classical GenEO coarse space is built from the following GEVP:
	\begin{align}
	\label{eq:eigGenEOClassis}
	\begin{array}{c}
	\text{Find } (\mathbf{V}_{jk},\lambda_{jk}) \in \R^{n_j} \setminus \{0\} \times \R
	\mbox{ such that}\\[1ex]
	 D_j R_j A R_j^T D_j \mathbf{V}_{jk} = \lambda_{jk} A_j^\mathrm{Neu} \mathbf{V}_{jk}\,,
	\end{array}
	\end{align}
 and selecting eigenvectors with $\lambda_{jk} > \tau$. Our experience in this instance is that there are too many large eigenvalues such that the eigensolver has difficulty finding them. This is addressed by the filtering technique introduced above. As a result, the cost remains comparable to that of classical GenEO approaches applied to scalar problems, such as Darcy flow or elasticity. In this context, the overhead is not a practical limitation, even in large-scale runs (see also~\cite{jolivet2013scalable,nataf2023geneo}). Each eigenproblem is local to a subdomain, thus embarrassingly parallel, and its cost depends only on the subdomain size, which can be adjusted by the user. 
\end{remark}

In the theory that follows for the stable decomposition estimate but not in the algorithm itself, we will make use of the projection $\pi_j$ which we now define.

\begin{definition}[Projection onto GenEO eigenspace]\label{def:proj_pi}
	Let $\pi_j$ be the projection from $R^{n_j}$ on
	\begin{align*}
		V_{j,\tau} \vcentcolon= \mathop\mathrm{span} \left\{ \mathbf{V}_{jk} |\, \lambda_{jk} > \tau \right\}
	\end{align*}
	parallel to the complement $\mathop\mathrm{span}\{ \mathbf{V}_{jk} |\, \lambda_{jk} \le \tau \}$.
\end{definition}
The key to understanding \Cref{gevp:tauthresholdAS} is the following bound, which can be derived directly from \cite[Lemma~7.7]{Dolean:2015:IDDSiam}.
\begin{lemma}\label{lem:GEVPbound}
	Consider \Cref{gevp:tauthresholdAS} and the projection $\pi_{j}$ from \Cref{def:proj_pi}. Then for all $\mathbf{U}_j \in R^{n_j}$ we have
	\begin{align}
		\label{eq:gevpLowerBoundAS}
		\left(R_j A R_j^T D_j (I-\xi_{0j}) (I-\pi_j) \mathbf{U}_j, D_j (I-\xi_{0j}) (I-\pi_j) \mathbf{U}_j\right) &\le \tau\, \left(A_j^\mathrm{Neu} \mathbf{U}_j, \mathbf{U}_j\right),
	\end{align}
	where we reiterate that $\tau > 0$ is a user-defined threshold within the GEVP.
\end{lemma}

We can now build the coarse space $V_0$ from the split near-kernel space $V_{G}$ in \Cref{def:near-kernelCS} and the GenEO space $V_\mathrm{GenEO}^\tau$ in \Cref{gevp:tauthresholdAS} by defining the vector space sum
\begin{align}
	\label{eq:V0GenEOAS}
	V_0 \vcentcolon= V_G + V_\mathrm{GenEO}^\tau.
\end{align}
The coarse space $V_0$ is spanned by the columns of a full rank rectangular matrix $Z$ with $n_0$ columns. We can now define the abstract framework for the two-level additive Schwarz preconditioner with the split near-kernel and GenEO coarse space. The definition of $\HP$ and the bilinear form $b(\cdot,\cdot)$ follow the same as in \eqref{eq:H_D} and \eqref{eq:b}, only with the definition of the space $V_0$ being different. We denote by $\RL_\mathrm{AS,2} \colon \HP \rightarrow \HO$ the linear operator linking the new product Hilbert space $\HP$ to $\HO$, defined identically to that in \eqref{eq:RAS} except that $Z$ and $P_{0}$ now stem from the newly enhanced coarse space \eqref{eq:V0GenEOAS}. In order to apply the FSL (\Cref{FSL}) we must check the three key assumptions as before. Surjectivity and continuity of $\RL_\mathrm{AS,2}$ follows identically to that for $\RL_\mathrm{AS}$ as in \Cref{lem:surj_cont}, thus, we need only consider the stable decomposition property.

\begin{lemma}[Stable decomposition with $\RL_\mathrm{AS,2}$]\label{lem:stable2}
	The operator $\RL_\mathrm{AS,2} \colon \HP \rightarrow \HO$ verifies the stable decomposition property required by \Cref{FSL} where, for a user-defined threshold $\tau>0$, the constant can be taken as
	\begin{align}
		\label{eq:c_T-GenEO}
		c_T \vcentcolon= \frac{1}{1 + \MC \tau}.
	\end{align}
\end{lemma}
\begin{proof}	
	Consider $\mathbf{U} \in \HO$, then following a similar argument to before we have
	\begin{align*}
		\mathbf{U} = P_0 \mathbf{U} &+ (I - P_0) \sum_{i=1}^N R_i^T D_i R_i \mathbf{U} \\
		= P_0 \mathbf{U} &+ (I - P_0) \sum_{i=1}^N R_i^T D_i (I - \xi_{0i}) R_i \mathbf{U} + \underbrace{(I - P_0) \sum_{i=1}^N R_i^T D_i \xi_{0i} R_i \mathbf{U}}_{\mathop=0}\\
		= P_0 \mathbf{U} &+ (I - P_0) \sum_{i=1}^N R_i^T D_i (I - \xi_{0i}) (I - \pi_i) R_i \mathbf{U}\\ & + \underbrace{(I - P_0) \sum_{i=1}^N R_i^T D_i (I - \xi_{0i}) \pi_i R_i \mathbf{U}}_{\mathop=0}.
	\end{align*}
	Now the very last term is zero since $R_i^T D_i (I - \xi_{0i}) \pi_i R_i \mathbf{U} \in V_\mathrm{GenEO}^\tau \subset V_0$ for all $1 \le i \le N$. Let $\mathbf{U}_0 \in \R^{n_0}$ be such that $Z \mathbf{U}_0 = P_0\mathbf{U}$, then we can choose the decomposition
	\begin{align*}
		\mathbf{U} = \RL_\mathrm{AS,2}(\mathbf{U}_0, (D_i (I - \xi_{0i}) (I - \pi_i) R_i \mathbf{U})_{1 \leq i \leq N}).
	\end{align*}
	With this decomposition, using the GEVP bound \eqref{eq:gevpLowerBoundAS} from \Cref{lem:GEVPbound} and then \Cref{Lemma_k1}, we have
	\begin{align*}
		(E\mathbf{U}_0, \mathbf{U}_0) + \sum_{j=1}^N &\left(R_j A R_j^T D_j (I-\xi_{0j}) (I - \pi_j) R_j \mathbf{U} D_j (I-\xi_{0j}) (I - \pi_j) R_j \mathbf{U}\right)\\ &\le a(Z\mathbf{U}_0, Z\mathbf{U}_0) + \tau \sum_{j=1}^N \left(A_j^\mathrm{Neu} R_j \mathbf{U}, R_j \mathbf{U}\right)\\
		&\le a(P_0\mathbf{U}, P_0\mathbf{U}) + \tau \, \MC \, a(\mathbf{U}, \mathbf{U})\\
		&\le (1 + \MC \tau) \, a(\mathbf{U}, \mathbf{U}).
	\end{align*}
	Thus the stable decomposition property holds with constant \eqref{eq:c_T-GenEO}.
\end{proof}

This provides the necessary machinery to apply the FSL (\Cref{FSL}). Defining the resulting preconditioner as $M^{-1}_\mathrm{AS,2}$, namely being
\begin{align*}
	M^{-1}_\mathrm{AS,2} &= \RL_\mathrm{AS,2} B^{-1} \RL_\mathrm{AS,2}^* = Z (Z^T A Z)^{-1} Z^T + (I-P_0) \sum_{i=1}^N R_i^T B_i^{-1} R_i (I-P_0^T) \\
	&= Z (Z^T A Z)^{-1} Z^T + (I-P_0) \sum_{i=1}^N R_i^T (R_i A R_i^T)^{-1} R_i (I-P_0^T),
\end{align*}
we arrive at our main theorem as a corollary of the results just proved.

\begin{theorem}[Convergence of additive Schwarz with the split near-kernel and GenEO coarse space]
	For a given user-defined constant $\tau>0$, namely the eigenvalue threshold within \Cref{gevp:tauthresholdAS}, the spectral condition number estimate for the two-level additive Schwarz preconditioner $M^{-1}_\mathrm{AS,2}$, with coarse space stemming from the split near-kernel (\Cref{def:near-kernelCS}) and GenEO eigenvectors (\Cref{gevp:tauthresholdAS}), is given by
	\begin{align*}
		\kappa(M^{-1}_\mathrm{AS,2} A) \le (1 + \MC \tau) \NC.
	\end{align*}
\end{theorem}

\section{Numerical results}
\label{sec:numerical_results}

In this section we study two broad challenges, that of domains with a growing number of holes and that of increasing heterogeneity within the physical parameters $\varepsilon$ and $\mu$ in the problem, and compare both one-level and two-level methods along with AMS. We use different colours in our tabulated results to indicate the performance of a given method; namely, green for low and robust iteration counts, orange for relatively low but not robust iteration counts, and red for iteration counts which become large and lack robustness. The reader can find a summary of the trends observed in our experiments in \Cref{tab:summary} as part of the conclusions.

In \cref{sec:influence_of_geometry}, we first present weak scaling experiments for a reference test case consisting of a 3D beam with square cross section having size $\frac{N}{2} \times 1 \times 1$, where $N$ is the number of subdomains, which may additionally have holes; see \Cref{fig:holes} for the configuration with holes. In addition, we perform a strong scalability study on a 3D beam of fixed size, as reported in \Cref{tab:metis_partitioning_both}. We then study the dependence on $\gamma$ in \cref{sec:gamma}, where $\gamma$ varies from $10^{-5}$ to $10^{2}$ on a complex geometric configuration. Based on these results, we select $\gamma=10^{-3}$ for the other experiments provided since it is representative of a difficult case. In \cref{sec:heterogeneities}, we move to study the impact of varying coefficients $\varepsilon$ and $\mu$. More details on each experiment will be given below.

The beam is discretised using a Cartesian mesh with mesh spacing $h = \frac{1}{16}$, on which we use N\'ed\'elec edge elements of lowest order; this leads to linear systems of size ranging from $121,\!696$ dofs, for $N = 8$, to $3,\!869,\!472$ dofs, for $N = 256$. In all test configurations, the right-hand side of equation~\eqref{eq:positivemaxwell} is defined as a constant volume source term, $\boldsymbol{f}=(1,1,1)$. To solve the linear systems we use right-preconditioned GMRES\footnote{Unlike left-preconditioning, this ensures that we minimise (the residual) in the same norm, independent of the quality of the preconditioner, to give a reliable comparison in the challenging nearly singular cases.} along with a relative stopping criterion of $10^{-6}$. Aside from in the strong scalability study, the subdomain partitioning is uniform along the length of the beam, yielding subdomains with constant size as $N$ varies. One layer of elements is added in each direction to define the overlapping decomposition; this results in an overlap region of two elements width between neighbouring subdomains. The strong scalability tests are performed similarly but with an arbitrary mesh partitioning into subdomains given by METIS~\cite{karypis:1998:fast}. The GenEO portion of the coarse space is constructed by selecting eigenvectors of \eqref{eq:eigAtildeBAS} corresponding to eigenvalues larger than $\tau = 10$.

All our numerical tests were performed with the free open-source domain specific language FreeFEM~\cite{Hecht:2012:NDF} and for each test configuration we compare the iteration counts of the following methods:
\begin{itemize}
	\item AMS: the auxiliary-space Maxwell solver, originally introduced by Hiptmair and Xu \cite{Hiptmair:2007:NSP} and available via PETSc~\cite{petsc-user-ref};
	\item AS: the one-level additive Schwarz preconditioner;
	\item AS-SNK: the two-level Schwarz preconditioner with a coarse space consisting of the split near-kernel (SNK) from \Cref{def:near-kernelCS};
	\item AS-SNK-GenEO: the two-level Schwarz preconditioner with a coarse space consisting of SNK plus $V_\mathrm{GenEO}^\tau$ with the eigenvectors computed by \eqref{eq:eigAtildeBAS};
	\item AS-NK: the two-level Schwarz preconditioner with a coarse space consisting of the ``global'' near-kernel (NK), i.e., simply $G$ from \Cref{def:near-kernelCS};
	\item AS-NK-GenEO: the two-level Schwarz preconditioner with a coarse space consisting of NK plus $V_\mathrm{GenEO}^\tau$.
\end{itemize}

Note that AS-SNK can be seen as a generalisation of the Nicolaides coarse space for the Poisson problem \cite{nicolaides1987deflation} to the positive Maxwell system. Indeed, for a Poisson problem, the near-kernel is made of constant functions and the Nicolaides coarse space consists in splitting it amongst the subdomains and multiplying each local component by the partition of unity (PoU). For a positive Maxwell problem in trivial topology, the near-kernel is made of the gradients of $H^1$ functions and the SNK coarse space also splits this amongst the subdomains and multiplies by the PoU, hence the name SNK. Theoretical results can be found in \cite{liang2021preconditioners,liang:2023:sharp} and references therein. The AS-SNK-GenEO approach is the full two-level preconditioner $M^{-1}_{AS,2}$ with the GenEO-enriched coarse space we have designed and analysed in this work. For comparison, we also report iteration counts for the latter two variants of the two-level preconditioner listed above which are not covered by our theory: AS-NK and AS-NK-GenEO. In each table exhibiting scalability results, we report the number of unknowns (dofs) of the linear system, the size of the ``global'' near-kernel (NK), the size of the split near-kernel (SNK), and the global number of GenEO eigenvectors entering the coarse space.

\subsection{Influence of the geometry}
\label{sec:influence_of_geometry}

In a first test configuration, with results collected in \Cref{tab:homogeneous_beam_both} (top), we choose the constant coefficients $\mu = \varepsilon = 1$ and enforce homogeneous Dirichlet boundary conditions $\mathbf{E} \times \mathbf{n} = \mathbf{0}$ on all six faces of the beam. We can see that all methods are robust, including the one-level method (AS) due to the Dirichlet boundary conditions surrounding the elongated geometry. Note that the adaptive GenEO coarse space is empty as it automatically captures the fact that there is no need to enrich the near-kernel coarse space to reach the prescribed condition number.

\begin{table}[ht!]
	\label{tab:homogeneous_beam_both}
	\begin{center}
		\caption{A weak scalability study detailing iteration counts required for the homogeneous beam problem with Dirichlet (top) or mixed (bottom) boundary conditions.}
		\begin{tabular}{l|l|r|r|r|r|r|r}
			& \hfill $N$                       &   8  &   16 &   32 &   64 &  128  &   256 \\ \hline
			& \#dofs                           & 122K & 243K & 484K & 968K & 1935K & 3869K \\
			& NK size                          & 19K  & 37K  & 74K  & 148K & 296K  & 592K  \\
			& SNK size                         & 25K  & 50K  & 101K & 203K & 406K  & 813K  \\
			& GenEO size (Dirichlet)           & 0    & 0    & 0    & 0    & 0     & 0     \\
			& GenEO size (mixed)               & 0    & 0    & 0    & 0    & 0     & 0     \\ \hline\hline
			\multirow[b]{4}{*}{\rotatebox[origin=c]{90}{\quad Dirichlet}}
			& \ccg AMS           & \ccg 7   & \ccg 7   & \ccg 7   & \ccg 7   & \ccg 7    & \ccg 7    \\
			& \ccg AS            & \ccg 12  & \ccg 15  & \ccg 15  & \ccg 15  & \ccg 15   & \ccg 14   \\
			& \ccg AS-SNK        & \ccg 14  & \ccg 15  & \ccg 15  & \ccg 15  & \ccg 15   & \ccg 15   \\
			& \ccg AS-SNK-GenEO  & \ccg 14  & \ccg 15  & \ccg 15  & \ccg 15  & \ccg 15   & \ccg 15   \\
			& \ccg AS-NK         & \ccg 14  & \ccg 15  & \ccg 15  & \ccg 15  & \ccg 15   & \ccg 15   \\
			& \ccg AS-NK-GenEO   & \ccg 14  & \ccg 15  & \ccg 15  & \ccg 15  & \ccg 15   & \ccg 15   \\ \hline\hline
			\multirow[b]{4}{*}{\rotatebox[origin=c]{90}{\quad mixed}}
			& \ccg AMS           & \ccg 11  & \ccg  11 & \ccg  11 & \ccg  12 & \ccg  12  & \ccg  12  \\
			& \ccr AS            & \ccr 20  & \ccr  27 & \ccr  41 & \ccr  69 & \ccr 124  & \ccr  231 \\
			& \ccg AS-SNK        & \ccg 15  & \ccg  16 & \ccg  17 & \ccg  17 & \ccg  18  & \ccg  18  \\
			& \ccg AS-SNK-GenEO  & \ccg 15  & \ccg  16 & \ccg  17 & \ccg  17 & \ccg  18  & \ccg  18  \\
			& \cco AS-NK         & \cco 19  & \cco  24 & \cco  32 & \cco  41 & \cco  42  & \cco  42  \\
			& \cco AS-NK-GenEO   & \cco 19  & \cco  24 & \cco  32 & \cco  41 & \cco  42  & \cco  42  \\
		\end{tabular}
	\end{center}
\end{table}

In a second test configuration, Neumann boundary conditions are imposed on two of the four lateral faces of the beam (chosen to be opposing) and the other faces of the beam retain a Dirichlet boundary condition. Results using these mixed boundary conditions are reported in \Cref{tab:homogeneous_beam_both} (bottom) where we can see that AMS is still robust, while the one-level method is not, with iteration counts ranging from 20 to 231. The SNK coarse space is sufficient to have a robust method, albeit with the iteration count mildly increasing from 15 to 18, so that again the GenEO coarse space is empty. The increase is a little more marked for the global NK coarse space, with iteration counts ranging from 19 to 42, suggesting that treatment, such as the splitting we perform for the SNK, is necessary to obtain a fully robust preconditioner. Nonetheless, note that the size of the NK coarse space is smaller than that of the SNK coarse space and that iteration counts appear to stabilise in the low forties, which is not prohibitively high.

\begin{figure}[t]
	\centering
	\includegraphics[width=0.7\textwidth,clip=true, trim = 0cm 7.4cm 0cm 7.4cm]{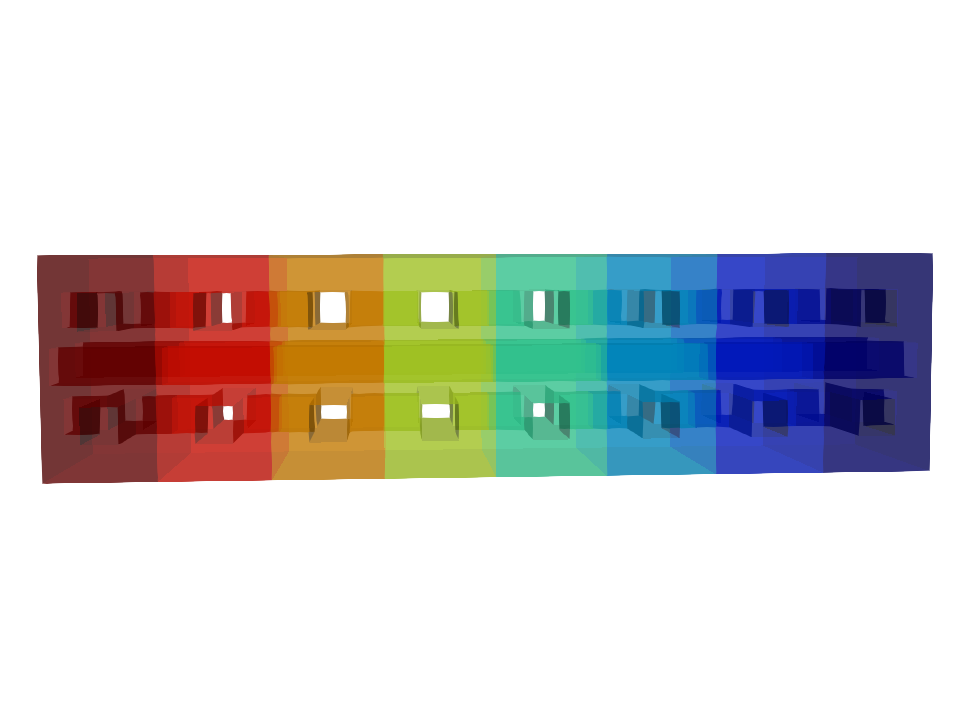}
	\includegraphics[width=0.28\textwidth]{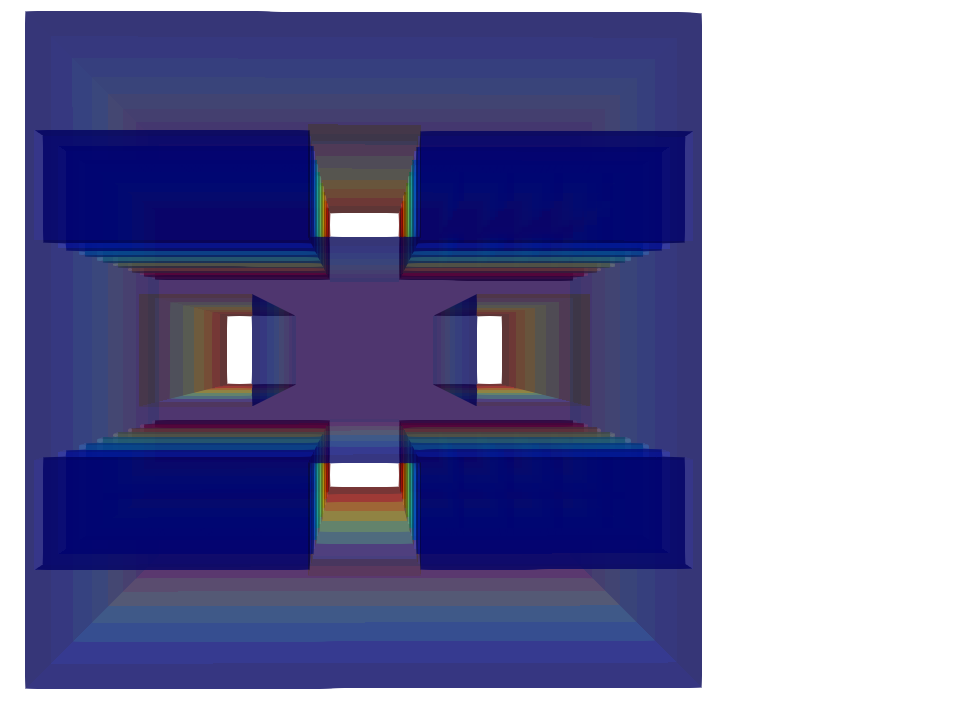}
	\caption{The beam test configuration with holes for $N=8$ subdomains, showing views from the side (left) and front (right) of the beam. The subdomain partitioning is also shown on the left with subdomains indicated by the different colours.}
	\label{fig:holes}
\end{figure}

\begin{table}[t!]
	\label{tab:beam_holes_both}
	\begin{center}
		\caption{A weak scalability study detailing iteration counts required for the homogeneous beam problem with holes (non-trivial topology) and Dirichlet (top) or mixed (bottom) boundary conditions.}
		\begin{tabular}{l|l|r|r|r|r|r|r}
			& \hfill $N$                       &   8  &   16 &   32 &   64 &  128  &   256 \\ \hline
			& \#dofs                           & 113K & 226K & 451K & 901K & 1800K & 3600K \\
			& NK size                          & 18K  & 36K  & 72K  & 144K & 288K  & 576K  \\
			& SNK size                         & 24K  & 49K  & 99K  & 198K & 397K  & 794K  \\
			& GenEO size (Dirichlet)           & 0    & 0    & 0    & 0    & 0     & 0     \\
			& GenEO size (mixed)               &   18 &   42 &  90  &  186 &   378 &  762  \\ \hline\hline
			\multirow[b]{4}{*}{\rotatebox[origin=c]{90}{\quad Dirichlet}}
			& \ccr AMS           & \ccr  44 & \ccr  71 & \ccr 169 & \ccr 364 & \ccr 660 & \ccr 813 \\
			& \cco AS           	& \cco  24 & \cco  36 & \cco  59 & \cco  79 & \cco  72 & \cco  58 \\
			& \ccg AS-SNK        & \ccg  14 & \ccg  15 & \ccg  15 & \ccg  15 & \ccg  15 & \ccg  15 \\
			& \ccg AS-SNK-GenEO  & \ccg  14 & \ccg  15 & \ccg  15 & \ccg  15 & \ccg  15 & \ccg  15 \\
			& \ccg AS-NK         & \ccg  17 & \ccg  18 & \ccg  18 & \ccg  18 & \ccg  17 & \ccg  17 \\
			& \ccg AS-NK-GenEO   & \ccg  17 & \ccg  18 & \ccg  18 & \ccg  18 & \ccg  17 & \ccg  17 \\ \hline\hline
			\multirow[b]{4}{*}{\rotatebox[origin=c]{90}{\quad mixed}}
			& \ccr AMS           & \ccr  43 & \ccr  67 & \ccr 113 & \ccr 321 & \ccr 588 & \ccr 1302\\
			& \ccr AS            & \ccr  37 & \ccr  61 & \ccr 106 & \ccr 173 & \ccr 294 & \ccr 557 \\
			& \ccr AS-SNK        & \ccr  36 & \ccr  62 & \ccr 109 & \ccr 202 & \ccr 383 & \ccr 554 \\
			& \ccg AS-SNK-GenEO  & \ccg  23 & \ccg  24 & \ccg  25 & \ccg  26 & \ccg  27 & \ccg  27 \\
			& \ccr AS-NK         & \ccr  36 & \ccr  62 & \ccr 111 & \ccr 203 & \ccr 386 & \ccr 588 \\
			& \cco AS-NK-GenEO   & \cco  24 & \cco  25 & \cco  25 & \cco  28 & \cco  30 & \cco  37 \\
		\end{tabular}
	\end{center}
\end{table}

\begin{figure}[t]
	\centering
	\includegraphics[width=0.63\textwidth,clip=true]{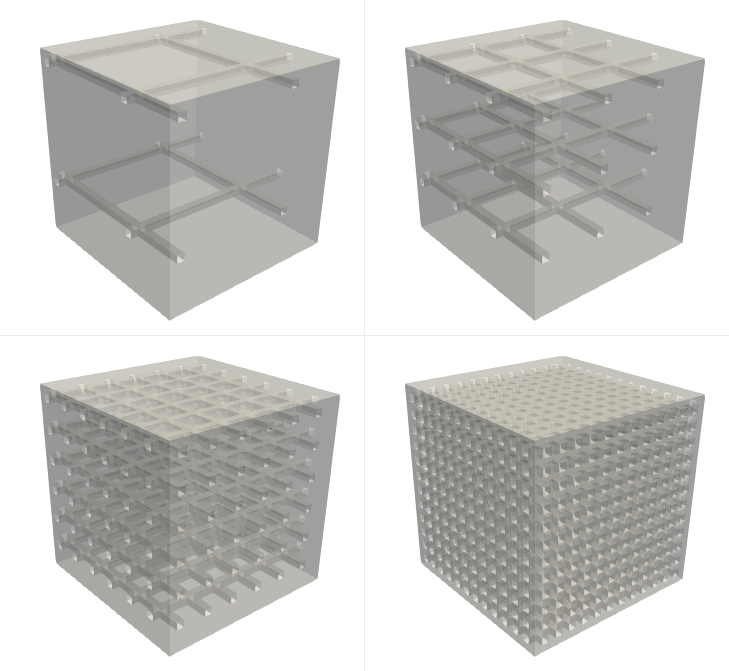}
	\caption{The unit cube test configuration with an increasing number of crossing holes in x-y directions: 8 (top left), 18 (top right), 72 (bottom left), and 288 (bottom right) holes.}
	\label{fig:cubes}
\end{figure}

For a third test configuration, we add various holes that extend across the domain. Four square holes extend along the whole length of the beam; see \Cref{fig:holes} (right). Further, two square holes perpendicularly cross each subdomain, from one lateral side to the opposite side, and also connect to two of the four holes along the length; see \Cref{fig:holes} (left). Since we perform a weak scalability test, the number of holes increases with the number of subdomains $N$. Homogeneous Dirichlet boundary conditions $\mathbf{E} \times \mathbf{n} = \mathbf{0}$ are enforced on the whole boundary of the domain. Results are gathered in \Cref{tab:beam_holes_both} (top). As we can see, with this configuration the AMS preconditioner is no longer robust, resulting in a significant degradation in performance, with iteration counts increasing from 44 to 813. Note that, in this case, we both elongate the domain while increasing the complexity of the topology. To further analyse this behaviour, we consider a fixed cube but with more and more holes in it; see \Cref{fig:cubes}. Results are detailed in \Cref{tab:cube}. The AMS approach is now scalable again and the decrease in the iteration counts from 26 to 18 is in agreement with the decrease in the problem size from 881k dofs to 696k dofs. These last two tables are in agreement with the theory of the AMS preconditioner, for which it is proved in \cite{Hiptmair:2020:review} to be robust with respect to the topology for a fixed sized domain; cf. \Cref{tab:cube}. But such a robustness result was not proved when both the domain is elongated and the topology complexity is increased. Our numerical results in \Cref{tab:beam_holes_both} suggest that it is indeed not the case. A possible fix might be to compute the full kernel of the curl operator for the geometry with holes which is larger than the gradients of $H^1$ functions. This can be done using homology computation tools, see \cite{Pellikka:2013:HCC}, but remains computationally expensive and it is unclear how such an approach can be incorporated into the AMS preconditioner. As for the one-level AS preconditioner, it also shows a lack of robustness here, while all two-level domain decomposition methods display nearly constant iteration counts and accordingly the GenEO coarse space is empty.

\begin{table}[t!]
	\begin{center}
		\caption{A study detailing iteration counts required for the unit cube problem with an increasing number of holes (see \Cref{fig:cubes}) and Dirichlet boundary conditions. METIS is used to partition the domain into $N=64$ subdomains.}
		\label{tab:cube}
		\begin{tabular}{l|r|r|r|r}
			\hfill \#holes                   & 8    & 18   & 72   & 288  \\ \hline
			\#dofs                           & 891K & 883K & 842K & 696K \\
			NK size                          & 132K & 132K & 129K & 120K \\
			SNK size                         & 221K & 219K & 210K & 185K \\
			GenEO size                       & 0    & 0    & 0    & 0    \\ \hline\hline
			\rowcolor{green!40}AMS           & 26   & 19   & 15   & 18   \\
			\rowcolor{orange!50}AS           & 83   & 71   & 69   & 61   \\
			\rowcolor{green!40}AS-SNK        & 36   & 33   & 32   & 28   \\
			\rowcolor{green!40}AS-SNK-GenEO  & 36   & 33   & 32   & 28   \\
			\rowcolor{orange!50}AS-NK        & 50   & 47   & 49   & 42   \\
			\rowcolor{orange!50}AS-NK-GenEO  & 50   & 47   & 49   & 42   \\
		\end{tabular}
	\end{center}
\end{table}

We return to considering the weak scaling setup with holes, as with \Cref{tab:beam_holes_both} (top), but now change the boundary conditions imposed on the boundaries created by holes from Dirichlet to Neumann. As shown in \Cref{tab:beam_holes_both} (bottom), for this configuration the AMS preconditioner is again no longer robust, resulting in a significant degradation in its performance, with iteration counts rising from 43 to 1302. The one-level AS preconditioner shows a similar lack of robustness, with iteration counts reaching 557. In contrast, the full AS-SNK-GenEO preconditioner shows nearly constant iteration counts, increasing only from 23 to 27, and hence good robustness. In this scenario, with mixed boundary conditions, the near-kernel coarse space needs to be enriched and GenEO selects around three vectors per subdomain to enter the coarse space. As for the two other variants, we can see that the global near-kernel (NK) alone is not enough. When adding the GenEO eigenvectors this leads to much better results but with iteration counts a little higher than with AS-SNK-GenEO and somewhat less favourable robustness.

In \Cref{tab:metis_partitioning_both} we move to consider a strong scalability study using a beam of size $8 \times 1 \times 1$, with the same hole configuration as in \Cref{tab:beam_holes_both} but note that now the number of holes is fixed. We vary the number of subdomains from $N=8$ to $N=256$ with this fixed problem size and geometry. In contrast to our previous experiments, where strip-wise partitioning was performed, here we use the automatic graph partitioner METIS~\cite{karypis:1998:fast} to decompose the mesh. We observe in case of Dirichlet conditions that, due to the fixed number of holes compared to the weak scaling experiment of \Cref{tab:beam_holes_both}, the iteration counts of AMS are stable but still relatively high (between 68 and 99). Moreover, the only scalable methods with low iteration counts are AS-SNK and AS-SNK-GenEO, which utilise the split near-kernel; GenEO not being required here.

\begin{table}[ht!]
	\begin{center}
		\caption{A strong scalability study detailing iteration counts required for the homogeneous beam problem of size $8 \times 1 \times 1$ with holes (non-trivial topology) and Dirichlet (top) or mixed (bottom) boundary conditions. The geometry is fixed and a METIS partitioning is used to define $N$ subdomains.}
		\label{tab:metis_partitioning_both}
		\begin{tabular}{l|l|r|r|r|r|r|r}
			& \hfill $N$                       &   8  &   16 &   32 &   64 &  128  &   256 \\ \hline
			& \#dofs                           & 226K & 226K & 226K & 226K & 226K  &  226K \\
			& NK size                          & 36K  & 36K  & 36K  & 36K  & 36K   &  36K  \\
			& SNK size                         & 43K  & 50K  & 58K  & 66K  & 82K   & 101K  \\
			& GenEO size (Dirichlet)           & 0    & 0    & 0    & 0    & 0     & 0     \\
			& GenEO size (mixed)               & 32   & 61   & 49   & 43   & 36    & 10    \\ \hline\hline
			\multirow[b]{4}{*}{\rotatebox[origin=c]{90}{\quad Dirichlet}}
			& \cco AMS             & \cco 97   & \cco 71   & \cco 68   & \cco 71   & \cco 77   & \cco 99    \\
			& \ccr AS              & \ccr 52   & \ccr 83   & \ccr 111  & \ccr 121  & \ccr 126  & \ccr 136   \\
			& \ccg AS-SNK          & \ccg 19   & \ccg 21   & \ccg 24   & \ccg 25   & \ccg 23   & \ccg 23    \\
			& \ccg AS-SNK-GenEO    & \ccg 19   & \ccg 21   & \ccg 24   & \ccg 25   & \ccg 23   & \ccg 23    \\
			& \cco AS-NK           & \cco 21   & \cco 24   & \cco 28   & \cco 30   & \cco 31   & \cco 32    \\
			& \cco AS-NK-GenEO     & \cco 21   & \cco 24   & \cco 28   & \cco 30   & \cco 31   & \cco 32    \\ \hline\hline
			\multirow[b]{4}{*}{\rotatebox[origin=c]{90}{\quad mixed}}
			& \ccr AMS             & \ccr 60   & \ccr 67   & \ccr 72   & \ccr 98   & \ccr 106  & \ccr 104   \\
			& \ccr AS              & \ccr 65   & \ccr 92   & \ccr 104  & \ccr 121  & \ccr 163  & \ccr 200   \\
			& \cco AS-SNK          & \cco 64   & \cco 87   & \cco 71   & \cco 35   & \cco 40   & \cco 25    \\
			& \ccg AS-SNK-GenEO    & \ccg 26   & \ccg 29   & \ccg 30   & \ccg 29   & \ccg 26   & \ccg 25    \\
			& \ccr AS-NK           & \ccr 64   & \ccr 89   & \ccr 102  & \ccr 113  & \ccr 155  & \ccr 181   \\
			& \ccr AS-NK-GenEO     & \ccr 38   & \ccr 65   & \ccr 95   & \ccr 109  & \ccr 156  & \ccr 181   \\
		\end{tabular}
	\end{center}
\end{table}

When Neumann boundary conditions (instead of Dirichlet boundary conditions) are imposed on the holes, we first observe that the iteration counts of AMS and AS now show a more substantial growth. Again, the only scalable methods for increasing $N$ are AS-SNK and AS-SNK-GenEO, which incorporate the split near-kernel (unlike in the case of \Cref{tab:beam_holes_both} (bottom) where AS-SNK shows the same bad behaviour as AS and AS-NK). Note that here the iteration counts of AS-SNK improve as the subdomains get smaller and, accordingly, the contribution of GenEO decreases to only $10$ eigenvectors in total for 256 subdomains, whereas in \Cref{tab:beam_holes_both} (mixed) the size of the GenEO coarse space is $3(N-2)$. This can be explained by the fact that, previously, all interior subdomains contained the same number of holes as $N$ increased, which is not the case here since subdomains are getting smaller and so there are less holes contained in the subdomains. We also observe in \Cref{tab:metis_partitioning_both} (along with results in our other experiments) that the SNK space is strictly larger than the NK space, which contains only gradients of functions. As the number of subdomains increases, the relative size of SNK compared to NK also increases (recall that the dimension of NK is fixed here) and, at some point, the SNK coarse space is nearly sufficient by itself. This difference is illustrated by the iteration counts of AS-NK and AS-NK-GenEO in \Cref{tab:metis_partitioning_both} (bottom), which grow to 181 for $N=256$, and further reveals that the NK space alone does not contain the relevant information on the presence of holes.

\subsection{Varying the parameter $\gamma$}
\label{sec:gamma}

In this section, we now consider the effect of varying the parameter $\gamma$ in \eqref{eq:bilinearformposmax}, which controls the size of the zeroth-order term, from $10^{-5}$ to $10^{2}$ for a decomposition into $N=256$ subdomains. \Cref{tab:gamma_no_hole_Dirichlet} considers the same beam problem and setup used in \Cref{tab:homogeneous_beam_both} (top) with $N=256$; the size of SNK and NK spaces staying the same. We see that, although the problem gets closer to being singular, all methods perform very well even for very small values of $\gamma$.

\begin{table}[ht!]
	\begin{center}
		\caption{Iteration counts required when varying the parameter $\gamma$ for the homogeneous beam problem (no holes) with Dirichlet boundary conditions and $N=256$ subdomains. Here we have \#dofs $= 3869K$, NK size $= 592K$, and SNK size $= 813K$.}
		\label{tab:gamma_no_hole_Dirichlet}
		\begin{tabular}{l|r|r|r|r|r|r|r|r}
			\hfill $\gamma$                  & $10^{-5}$ & $10^{-4}$ & $10^{-3}$ & $10^{-2}$ & $10^{-1}$ & $1$ & $10$ & $10^{2}$ \\ \hline
			GenEO size                       & 0  & 0  & 0  & 0  & 0  & 0  & 0  & 0  \\ \hline\hline
			\rowcolor{green!40}AMS           & 7  & 7  & 7  & 7  & 7  & 8  & 7  & 5  \\
			\rowcolor{green!40}AS            & 14 & 14 & 14 & 14 & 14 & 14 & 11 & 6  \\
			\rowcolor{green!40}AS-SNK        & 15 & 15 & 15 & 15 & 15 & 14 & 12 & 9  \\
			\rowcolor{green!40}AS-SNK-GenEO  & 15 & 15 & 15 & 15 & 15 & 14 & 12 & 9  \\
			\rowcolor{green!40}AS-NK         & 15 & 15 & 15 & 15 & 15 & 14 & 12 & 10 \\
			\rowcolor{green!40}AS-NK-GenEO   & 15 & 15 & 15 & 15 & 15 & 14 & 12 & 10 \\
		\end{tabular}
	\end{center}
\end{table}

In \Cref{tab:gamma_holes_both} we now add in holes to the beam and consider both Dirichlet and mixed boundary conditions, as in \Cref{tab:beam_holes_both}, again focusing on $N=256$. The smaller the value $\gamma>0$ takes, the harder the problem is. For the Dirichlet case this is seen by the AMS and one-level AS methods struggling. Otherwise, all our two-level methods perform very well and, accordingly, the GenEO portion of the coarse space is empty.

\begin{table}[ht!]
	\begin{center}
		\caption{Iteration counts required when varying the parameter $\gamma$ for the homogeneous beam problem with holes (non-trivial topology), Dirichlet (top) or mixed (bottom) boundary conditions, and $N=256$ subdomains. Here we have \#dofs $= 3600K$, NK size $= 576K$, and SNK size $= 794K$.}
		\label{tab:gamma_holes_both}
		\begin{tabular}{l|l|r|r|r|r|r|r|r|r}
			\multicolumn{2}{l|}{\hfill $\gamma$} & $10^{-5}$ & $10^{-4}$ & $10^{-3}$ & $10^{-2}$ & $10^{-1}$ & $1$ & $10$ & $10^{2}$ \\ \hline
			\multicolumn{2}{l|}{GenEO size (Dirichlet)} & 0    & 0    & 0    & 0   & 0   & 0   & 0  & 0  \\
			\multicolumn{2}{l|}{GenEO size (mixed)}     & 762  & 762  & 762  & 762 & 762 & 762 & 0  & 0  \\ \hline\hline
			\multirow[b]{4}{*}{\rotatebox[origin=c]{90}{\quad Dirichlet}}
			& \ccr AMS           & \ccr 494  & \ccr 569  & \ccr 813  & \ccr 342 & \ccr 98  & \ccr 38  & \ccr 16 & \ccr 8  \\
			& \cco AS            & \cco 54   & \cco 54   & \cco 58   & \cco 84  & \cco 46  & \cco 23  & \cco 13 & \cco 7  \\
			& \ccg AS-SNK        & \ccg 15   & \ccg 15   & \ccg 15   & \ccg 15  & \ccg 15  & \ccg 14  & \ccg 12 & \ccg 9  \\
			& \ccg AS-SNK-GenEO  & \ccg 15   & \ccg 15   & \ccg 15   & \ccg 15  & \ccg 15  & \ccg 14  & \ccg 12 & \ccg 9  \\
			& \ccg AS-NK         & \ccg 17   & \ccg 17   & \ccg 17   & \ccg 17  & \ccg 16  & \ccg 14  & \ccg 12 & \ccg 9  \\
			& \ccg AS-NK-GenEO   & \ccg 17   & \ccg 17   & \ccg 17   & \ccg 17  & \ccg 16  & \ccg 14  & \ccg 12 & \ccg 9  \\ \hline\hline
			\multirow[b]{4}{*}{\rotatebox[origin=c]{90}{\quad mixed}}
			& \ccr AMS           & \ccr 2057 & \ccr 2193 & \ccr 1302 & \ccr 303 & \ccr 86  & \ccr 33  & \ccr 14 & \ccr 6  \\
			& \ccr AS            & \ccr 618  & \ccr 579  & \ccr 557  & \ccr 463 & \ccr 150 & \ccr 46  & \ccr 15 & \ccr 7  \\
			& \ccr AS-SNK        & \ccr 772  & \ccr 761  & \ccr 554  & \ccr 444 & \ccr 144 & \ccr 45  & \ccr 17 & \ccr 9  \\
			& \ccg AS-SNK-GenEO  & \ccg 28   & \ccg 27   & \ccg 27   & \ccg 26  & \ccg 23  & \ccg 20  & \ccg 17 & \ccg 9  \\
			& \ccr AS-NK         & \ccr 776  & \ccr 765  & \ccr 588  & \ccr 444 & \ccr 144 & \ccr 45  & \ccr 17 & \ccr 10 \\
			& \cco AS-NK-GenEO   & \cco 51   & \cco 49   & \cco 37   & \cco 25  & \cco 24  & \cco 20  & \cco 17 & \cco 10 \\
		\end{tabular}
	\end{center}
\end{table}

\Cref{tab:gamma_holes_both} (bottom) exhibits the mixed boundary conditions scenario with holes. The exacting challenge as $\gamma$ shrinks is reflected in the iteration counts, which fail to be robust for all methods except AS-SNK-GenEO, in which case the iteration counts tend to stabilise at around $28$ iterations, and to a lesser extent AS-NK-GenEO, where the iteration counts are under $60$. Note that for large values of $\gamma$ (namely $\gamma > 1$), the GenEO portion of the coarse space is empty, otherwise its size remains at a constant $762$ vectors for this particular problem.

\subsection{Heterogeneities}
\label{sec:heterogeneities}

In this section we detail experiments with heterogeneous problems similar to those in \cite{Hiptmair:2007:NSP}. The parameter $\gamma$ is kept constant at $\gamma=10^{-3}$ in order to retain a challenging problem. For the results in \Cref{tab:epsilon_alternating_both}, the beam is made of eight alternating layers with the parameter $\varepsilon$ alternating between a value of unity and a value $\varepsilon_\text{alt}$ which we vary, namely ranging between $10^{-4}$ and $10^4$, while $\mu$ is kept uniform throughout at a constant value $\mu=1$. In \Cref{tab:epsilon_alternating_both} (top), we consider Dirichlet boundary conditions on all six faces of the beam. We see that all methods perform very well and that, except for the first column, there is no contribution of the GenEO methodology to the coarse space.

\begin{table}[t!]
	\begin{center}
		\caption{Iteration counts required when varying $\varepsilon$ for the heterogeneous beam problem with alternating layers, Dirichlet (top) or mixed (bottom) boundary conditions, and $N=256$ subdomains. Here we have \#dofs $= 3869K$, NK size $= 592K$, and SNK size $= 813K$.}
		\label{tab:epsilon_alternating_both}
		\small
		\begin{tabular}{l|l|r|r|r|r|r|r|r|r|r}
			\multicolumn{2}{l|}{\hfill $\varepsilon_\text{alt}$} & $10^{-4}$ & $10^{-3}$ & $10^{-2}$ & $10^{-1}$ & $1$ & $10$ & $10^{2}$ & $10^{3}$ & $10^{4}$ \\ \hline
			\multicolumn{2}{l|}{GenEO size (Dirichlet)} & 9711 & 0  & 0  & 0  & 0  & 0  & 0  & 0  & 0  \\
			\multicolumn{2}{l|}{GenEO size (mixed)}     & 0   & 0   & 0   & 0   & 0   & 0   & 0   & 0  & 0  \\ \hline\hline
			\multirow[b]{4}{*}{\rotatebox[origin=c]{90}{\quad Dirichlet}}
			& \ccg AMS          & \ccg 7   & \ccg 7   & \ccg 7   & \ccg 7   & \ccg 7   & \ccg 7   & \ccg 7   & \ccg 8  & \ccg 10 \\
			& \ccg AS           & \ccg 14  & \ccg 14  & \ccg 14  & \ccg 14  & \ccg 14  & \ccg 14  & \ccg 15  & \ccg 15 & \ccg 17 \\
			& \ccg AS-SNK       & \ccg 15  & \ccg 15  & \ccg 15  & \ccg 15  & \ccg 15  & \ccg 15  & \ccg 15  & \ccg 15 & \ccg 15 \\
			& \ccg AS-SNK-GenEO & \ccg 15  & \ccg 15  & \ccg 15  & \ccg 15  & \ccg 15  & \ccg 15  & \ccg 15  & \ccg 15 & \ccg 15 \\
			& \ccg AS-NK        & \ccg 15  & \ccg 15  & \ccg 15  & \ccg 15  & \ccg 15  & \ccg 15  & \ccg 15  & \ccg 15 & \ccg 15 \\
			& \ccg AS-NK-GenEO  & \ccg 15  & \ccg 15  & \ccg 15  & \ccg 15  & \ccg 15  & \ccg 15  & \ccg 15  & \ccg 15 & \ccg 15 \\ \hline\hline
			\multirow[b]{4}{*}{\rotatebox[origin=c]{90}{\quad mixed}}
			& \ccg AMS          & \ccg 13  & \ccg 13  & \ccg 13  & \ccg 12  & \ccg 12  & \ccg 11  & \ccg 11  & \ccg 11 & \ccg 10 \\
			& \ccr AS           & \ccr 288 & \ccr 233 & \ccr 233 & \ccr 233 & \ccr 231 & \ccr 228 & \ccr 202 & \ccr 65 & \ccr 21 \\
			& \ccg AS-SNK       & \ccg 18  & \ccg 18  & \ccg 18  & \ccg 18  & \ccg 18  & \ccg 18  & \ccg 18  & \ccg 17 & \ccg 14 \\
			& \ccg AS-SNK-GenEO & \ccg 18  & \ccg 18  & \ccg 18  & \ccg 18  & \ccg 18  & \ccg 18  & \ccg 18  & \ccg 17 & \ccg 14 \\
			& \cco AS-NK        & \cco 43  & \cco 43  & \cco 43  & \cco 42  & \cco 42  & \cco 38  & \cco 33  & \cco 25 & \cco 16 \\
			& \cco AS-NK-GenEO  & \cco 43  & \cco 43  & \cco 43  & \cco 42  & \cco 42  & \cco 38  & \cco 33  & \cco 25 & \cco 16 \\
		\end{tabular}
	\end{center}
\end{table}

In \Cref{tab:epsilon_alternating_both} (bottom) we change the boundary conditions from Dirichlet everywhere to have Neumann conditions on two opposing lateral faces of the beam and, as $\varepsilon$ tends to zero, the problem demands a good preconditioner. AMS, AS-SNK and AS-SNK-GenEO all yield very robust iteration counts, however, AS-NK and AS-NK-GenEO require iterations counts that increase from $16$ to $43$ as $\varepsilon$ gets smaller. The one-level AS method is not robust at all, with iteration counts reaching $288$. Note that in this problem instance the GenEO coarse space is again empty and thus not required.

In the setup for \Cref{tab:mu_alternating_both}, the beam is still made of eight alternating layers but now $\varepsilon$ is fixed throughout at $\varepsilon=1$ and, instead, $\mu$ takes alternate values of unity and a value $\mu_\text{alt}$ between $10^{-4}$ and $10^4$. In \Cref{tab:mu_alternating_both} (top), Dirichlet boundary conditions are imposed on all faces of the beam whereas in \Cref{tab:mu_alternating_both} (bottom) they are again changed to Neumann boundary conditions on two opposing lateral faces of the beam. Numerical results suggest that both small and large extreme values for $\mu$ make the problem harder. In this scenario only AMS and the approaches enriched by GenEO, namely AS-SNK-GenEO and AS-NK-GenEO, have robust iteration counts. All other methods have iteration counts that can reach over $300$ for the hardest problems. In contrast to varying $\varepsilon$, we now see that the GenEO portion of the coarse space is essential and grows as the problem becomes more challenging, in particular when $\mu_\text{alt}$ increases to become very large. In our most extreme case of $\mu_\text{alt} = 10^{4}$ we require, on average, approximately $300$ eigenvectors per subdomain.

\begin{table}[ht!]
	\begin{center}
		\caption{Iteration counts required when varying $\mu$ for the heterogeneous beam problem with alternating layers, Dirichlet (top) or mixed (bottom) boundary conditions, and $N=256$ subdomains. Here we have \#dofs $= 3869K$, NK size $= 592K$, and SNK size $= 813K$. The out-of-memory (OOM) is not due to the size of the coarse space alone (indeed, enlarging it with the GenEO vectors removed this issue), but rather to limitations of the coarse solver used, which exceeded the available system memory in this case.}
		\label{tab:mu_alternating_both}
		\small
		\begin{tabular}{l|l|r|r|r|r|r|r|r|r|r}
			\multicolumn{2}{l|}{\hfill $\mu_\text{alt}$} & $10^{-4}$ & $10^{-3}$ & $10^{-2}$ & $10^{-1}$ & $1$ & $10$ & $10^{2}$ & $10^{3}$ & $10^{4}$ \\ \hline
			\multicolumn{2}{l|}{GenEO size (Dir.)} & 25K & 12K & 1272 & 0  & 0  & 0  & 1528 & 12K & 25K \\
			\multicolumn{2}{l|}{GenEO size (mixed)}     & 2560 & 2560 & 1272 & 0 & 0 & 0  & 1526 & 11K & 77K \\ \hline\hline
			\multirow[b]{4}{*}{\rotatebox[origin=c]{90}{\quad Dirichlet}}
			& \ccg AMS           & \ccg 13   & \ccg 12   & \ccg 11   & \ccg 9   & \ccg 7   & \ccg 9   & \ccg 11   & \ccg 12   & \ccg 14   \\
			& \ccr AS            & \ccr 394  & \ccr 213  & \ccr 76   & \ccr 25  & \ccr 14  & \ccr 25  & \ccr 76   & \ccr 211  & \ccr 380  \\
			& \ccr AS-SNK        & \ccr 41   & \ccr 42   & \ccr 34   & \ccr 21  & \ccr 15  & \ccr 22  & \ccr 54   & \ccr 150  & \ccr 335  \\
			& \ccg AS-SNK-GenEO  & \ccg 19   & \ccg 18   & \ccg 16   & \ccg 21  & \ccg 15  & \ccg 22  & \ccg 18   & \ccg 18   & \ccg 20   \\
			& \ccr AS-NK         & \ccr 46   & \ccr 44   & \ccr 34   & \ccr 21  & \ccr 15  & \ccr 22  & \ccr 55   & \ccr 154  & \ccr 340  \\
			& \ccg AS-NK-GenEO   & \ccg 32   & \ccg 19   & \ccg 17   & \ccg 21  & \ccg 15  & \ccg 22  & \ccg 18   & \ccg 19   & \ccg 20   \\ \hline\hline
			\multirow[b]{4}{*}{\rotatebox[origin=c]{90}{\quad mixed}}
			& \ccg AMS           & \ccg 13   & \ccg 13   & \ccg 13   & \ccg 13  & \ccg 12  & \ccg 12  & \ccg 12   & \ccg 13   & \ccg 14   \\
			& \ccr AS            & \ccr 581  & \ccr 407  & \ccr 275  & \ccr 236 & \ccr 231 & \ccr 233 & \ccr 264  & \ccr 378  & \ccr 523  \\
			& \ccr AS-SNK        & \ccr OOM  & \ccr 40   & \ccr 36   & \ccr 24  & \ccr 18  & \ccr 26  & \ccr 58   & \ccr 159  & \ccr 339  \\
			& \ccg AS-SNK-GenEO  & \ccg 19   & \ccg 19   & \ccg 18   & \ccg 24  & \ccg 18  & \ccg 26  & \ccg 19   & \ccg 19   & \ccg 20   \\
			& \ccr AS-NK         & \ccr 48   & \ccr 47   & \ccr 44   & \ccr 43  & \ccr 42  & \ccr 42  & \ccr 67   & \ccr 174  & \ccr 354  \\
			& \ccg AS-NK-GenEO   & \ccg 44   & \ccg 44   & \ccg 43   & \ccg 43  & \ccg 42  & \ccg 42  & \ccg 42   & \ccg 43   & \ccg 41   \\
		\end{tabular}
	\end{center}
\end{table}
In our final set of experiments we use a more geometrically complex heterogeneity. The problem is again defined in the whole beam but this time with $\varepsilon$ or $\mu$ varying from unity in the location of the holes used in previous tests, akin to those shown in \Cref{fig:holes} but now with a fixed $N=256$ subdomains. As before, we vary separately $\varepsilon$ and $\mu$ to distinguish their effects and label the value taken in the holes as $\varepsilon_\text{holes}$ or $\mu_\text{holes}$, respectively. We again consider pure Dirichlet and mixed boundary condition cases, as described previously. Results in \Cref{tab:epsilon_complex_both} detail the behaviour when $\varepsilon_\text{holes}$ varies between $10^{-4}$ and $10^4$ and $\mu$ is kept constant at $\mu=1$. We see that the performance of each method is rather similar to that observed in \Cref{tab:epsilon_alternating_both}, suggesting that the precise geometric variation of $\varepsilon$ is not too important.

\begin{table}[ht!]
	\begin{center}
		\caption{Iteration counts required when varying $\varepsilon$ for the heterogeneous beam problem with complex geometric heterogeneity which varies in the location of holes akin to \Cref{fig:holes}, Dirichlet (top) or mixed (bottom) boundary conditions, and $N=256$ subdomains. Here we have \#dofs $= 3869K$, NK size $= 592K$, and SNK size $= 813K$.}
		\label{tab:epsilon_complex_both}
		\small
		\begin{tabular}{l|l|r|r|r|r|r|r|r|r|r}
			\multicolumn{2}{l|}{\hfill $\varepsilon_\text{holes}$} & $10^{-4}$ & $10^{-3}$ & $10^{-2}$ & $10^{-1}$ & $1$ & $10$ & $10^{2}$ & $10^{3}$ & $10^{4}$ \\ \hline
			\multicolumn{2}{l|}{GenEO size (Dirichlet)} & 1023 & 0  & 0  & 0  & 0  & 0  & 0  & 0  & 0  \\
			\multicolumn{2}{l|}{GenEO size (mixed)}     & 0   & 0   & 0   & 0   & 0   & 0   & 0   & 0   & 0   \\ \hline\hline
			\multirow[b]{4}{*}{\rotatebox[origin=c]{90}{\quad Dirichlet}}
			& \ccg AMS           & \ccg 7   & \ccg 7   & \ccg 7   & \ccg 7   & \ccg 7   & \ccg 7   & \ccg 7   & \ccg 8   & \ccg 9   \\
			& \ccr AS            & \ccr 14  & \ccr 14  & \ccr 14  & \ccr 14  & \ccr 14  & \ccr 14  & \ccr 17  & \ccr 25  & \ccr 62  \\
			& \ccg AS-SNK        & \ccg 15  & \ccg 15  & \ccg 15  & \ccg 15  & \ccg 15  & \ccg 15  & \ccg 15  & \ccg 15  & \ccg 15  \\
			& \ccg AS-SNK-GenEO  & \ccg 15  & \ccg 15  & \ccg 15  & \ccg 15  & \ccg 15  & \ccg 15  & \ccg 15  & \ccg 15  & \ccg 15  \\
			& \ccg AS-NK         & \ccg 15  & \ccg 15  & \ccg 15  & \ccg 15  & \ccg 15  & \ccg 15  & \ccg 15  & \ccg 15  & \ccg 15  \\
			& \ccg AS-NK-GenEO   & \ccg 15  & \ccg 15  & \ccg 15  & \ccg 15  & \ccg 15  &\ccg  15  & \ccg 15  & \ccg 15  & \ccg 15  \\ \hline\hline
			\multirow[b]{4}{*}{\rotatebox[origin=c]{90}{\quad mixed}}
			& \ccg AMS           & \ccg 12  & \ccg 12  & \ccg 12  & \ccg 12  & \ccg 12  & \ccg 12  & \ccg 13  & \ccg 12  & \ccg 12  \\
			& \ccr AS            & \ccr 232 & \ccr 232 & \ccr 232 & \ccr 232 & \ccr 231 & \ccr 231 & \ccr 226 & \ccr 180 & \ccr 135 \\
			& \ccg AS-SNK        & \ccg 18  & \ccg 18  & \ccg 18  & \ccg 18  & \ccg 18  & \ccg 17  & \ccg 17  & \ccg 16  & \ccg 15  \\
			& \ccg AS-SNK-GenEO  & \ccg 18  & \ccg 18  & \ccg 18  & \ccg 18  & \ccg 18  & \ccg 17  & \ccg 17  & \ccg 16  & \ccg 15  \\
			& \cco AS-NK         & \cco 42  & \cco 42  & \cco 42  & \cco 42  & \cco 42  & \cco 41  & \cco 38  & \cco 33  & \cco 24  \\
			& \cco AS-NK-GenEO   & \cco 42  & \cco 42  & \cco 42  & \cco 42  & \cco 42  & \cco 41  & \cco 38  & \cco 33  & \cco 24  \\
		\end{tabular}
	\end{center}
\end{table}

In contrast, when varying $\mu_\text{holes}$ between $10^{-4}$ and $10^4$ in \Cref{tab:mu_complex_both}, with a fixed $\varepsilon=1$, we now observe some degradation in the performance of AMS, suggesting that this approach is not fully robust to complex heterogeneities. Our approach, AS-SNK-GenEO, yields the lowest iteration counts in the most challenging regime of large $\mu_\text{holes}$ and provides good behaviour overall, with iteration counts stabilising at around $25$ once the GenEO space is required. Note that the size of the GenEO coarse space adapts to the difficulty of the problem and thus ultimately increases with the value of $\mu_\text{holes}$. More precisely, the number of GenEO vectors is zero for moderate $\mu_\text{holes}$ until it reaches $10^2$, at which point the size of the GenEO space increases up to $15K$ once $\mu_\text{holes} = 10^{4}$, which equates to $62$ eigenvectors per subdomain on average. While AS-NK-GenEO also performs well in terms of robustness to varying $\mu$, both the one-level AS method and those without the enrichment of GenEO, that is AS-SNK and AS-NK, become much poorer for large values of $\mu_\text{holes}$.

\begin{table}[t!]
	\begin{center}
		\caption{Iteration counts required when varying $\mu$ for the heterogeneous beam problem with complex geometric heterogeneity which varies in the location of holes akin to \Cref{fig:holes}, Dirichlet (top) or mixed (bottom) boundary conditions, and $N=256$ subdomains. Here we have \#dofs $= 3869K$, NK size $= 592K$, and SNK size $= 813K$.}
		\label{tab:mu_complex_both}
		\small
		\begin{tabular}{l|l|r|r|r|r|r|r|r|r|r}
			\multicolumn{2}{l|}{\hfill $\mu_\text{holes}$} & $10^{-4}$ & $10^{-3}$ & $10^{-2}$ & $10^{-1}$ & $1$ & $10$ & $10^{2}$ & $10^{3}$ & $10^{4}$ \\ \hline
			\multicolumn{2}{l|}{GenEO size (Dirichlet)} & 25K & 5610 & 0  & 0   & 0   & 0   & 762 & 762  & 15K \\
			\multicolumn{2}{l|}{GenEO size (mixed)}     & 0   & 0   & 0   & 0   & 0   & 0   & 762 & 1016 & 15K \\ \hline\hline
			\multirow[b]{4}{*}{\rotatebox[origin=c]{90}{\quad Dirichlet}}
			& \cco AMS          & \cco 13  & \cco 12  & \cco 11  & \cco 8   & \cco 7   & \cco 10  & \cco 22  & \cco 56   & \cco 158   \\
			& \ccr AS           & \ccr 16  & \ccr 16  & \ccr 16  & \ccr 16  & \ccr 14  & \ccr 17  & \ccr 35  & \ccr 108  & \ccr 333   \\
			& \ccr AS-SNK       & \ccr 15  & \ccr 15  & \ccr 15  & \ccr 15  & \ccr 15  & \ccr 17  & \ccr 35  & \ccr 105  & \ccr 322   \\
			& \ccg AS-SNK-GenEO & \ccg 12  & \ccg 11  & \ccg 15  & \ccg 15  & \ccg 15  & \ccg 17  & \ccg 18  & \ccg 23   & \ccg 25    \\
			& \ccr AS-NK        & \ccr 15  & \ccr 15  & \ccr 15  & \ccr 15  & \ccr 15  & \ccr 17  & \ccr 35  & \ccr 105  & \ccr 322   \\
			& \ccg AS-NK-GenEO  & \ccg 17  & \ccg 13  & \ccg 15  & \ccg 15  & \ccg 15  & \ccg 17  & \ccg 18  & \ccg 23   & \ccg 25    \\ \hline\hline
			\multirow[b]{4}{*}{\rotatebox[origin=c]{90}{\quad mixed}}
			& \cco AMS          & \cco 18  & \cco 18  & \cco 16  & \cco 13  & \cco 12  & \cco 12  & \cco 20  & \cco 32   & \cco 43    \\
			& \ccr AS           & \ccr 216 & \ccr 216 & \ccr 216 & \ccr 219 & \ccr 231 & \ccr 241 & \ccr 247 & \ccr 299  & \ccr 448   \\
			& \ccr AS-SNK       & \ccr 18  & \ccr 18  & \ccr 18  & \ccr 18  & \ccr 18  & \ccr 23  & \ccr 45  & \ccr 119  & \ccr 287   \\
			& \ccg AS-SNK-GenEO & \ccg 18  & \ccg 18  & \ccg 18  & \ccg 18  & \ccg 18  & \ccg 23  & \ccg 26  & \ccg 24   & \ccg 25    \\
			& \ccr AS-NK        & \ccr 39  & \ccr 39  & \ccr 39  & \ccr 39  & \ccr 42  & \ccr 44  & \ccr 53  & \ccr 132  & \ccr 299   \\
			& \ccg AS-NK-GenEO  & \ccg 39  & \ccg 39  & \ccg 39  & \ccg 39  & \ccg 42  & \ccg 44  & \ccg 45  & \ccg 45   & \ccg 43    \\
		\end{tabular}
	\end{center}
\end{table}

\section{Conclusions}
\label{sec:conclusions}

In this work, we have both designed and tested a new two-level domain decomposition preconditioner for positive Maxwell equations inspired by the idea of subspace decomposition, but based on spectral coarse spaces. This was motivated by the AMS preconditioner of Hiptmair and Xu, although our approach is different in spirit. Our method performs well irrespective of the topology of the domain and which boundary conditions are imposed, as can be seen in a summary of trends exhibited by our experiments in \Cref{tab:summary}, where only our approach performs well in all scenarios. The numerical results corroborate our theory and show robustness also with respect to heterogeneous coefficients.

It is worth emphasising that the AMS preconditioner is highly efficient when applied to problems with simple geometries and no topological complexity (e.g., without holes). In such scenarios, AMS may well be preferable to domain decomposition methods, as it benefits from optimised implementations available in libraries such as PETSc. However, the performance of AMS as well as related multigrid techniques depend heavily on the choice of prolongation operators used, which in turn depend on the discretisation scheme. This can limit their applicability in evolving or heterogeneous simulation environments. In contrast, the AS-SNK-GenEO method is designed for robustness precisely in more challenging regimes, such as non-trivial topologies, mixed boundary conditions, and strong heterogeneities, where AMS may deteriorate significantly. A key strength of GenEO-based methods lies in their flexibility: the coarse space is constructed automatically through local generalised eigenvalue problems, adapting naturally to the problem at hand and to the discretisation. This makes our approach particularly useful in scenarios where no corresponding AMS theory or robust implementation exists, or when the discretisation scheme may change over the lifetime of the software. In this sense, AMS and GenEO-type methods are complementary, and hybrid or adaptive strategies could be explored in future work to combine their strengths.

\begin{table}[ht!]
	\begin{center}
		\caption{A summary of the trends exhibited in our numerical results for the different preconditioners. Recall that green indicates low and robust iteration counts, orange implies relatively low but not robust iteration counts, while red signals that iteration counts become large and lack robustness. Note: AMS is highly effective for simple geometries without holes, benefiting from available optimised implementations. In contrast, GenEO-type methods adapt automatically to the discretisation and are more robust in complex or evolving simulation settings.}
		\label{tab:summary}
		\small
		\begin{tabular}{l|l|l|c|c|c|c|c|c|c}
			Experiment & Geometry & BCs & Table & \rotatebox[origin=l]{90}{AMS} & \rotatebox[origin=l]{90}{AS} & \rotatebox[origin=l]{90}{AS-SNK} & \rotatebox[origin=l]{90}{AS-SNK-GenEO} & \rotatebox[origin=l]{90}{AS-NK} & \rotatebox[origin=l]{90}{AS-NK-GenEO} \\ \hline\hline
			\multirow{4}{*}{weak scalability} & \multirow{2}{*}{no holes} & Dirichlet & \multirow{2}{*}{\ref{tab:homogeneous_beam_both}} & \ccg & \ccg & \ccg & \ccg & \ccg & \ccg \\
			& & mixed & & \ccg & \ccr & \ccg & \ccg & \cco & \cco \\
			\cline{2-9}
			& \multirow{2}{*}{holes} & Dirichlet & \multirow{2}{*}{\ref{tab:beam_holes_both}} & \ccr & \cco & \ccg & \ccg & \ccg & \ccg \\
			& & mixed & & \ccr & \ccr & \ccr & \ccg & \ccr & \cco \\
			\hline
			increasing holes & fixed domain & Dirichlet & \ref{tab:cube} & \ccg & \cco & \ccg & \ccg & \cco & \cco \\
			\hline
			\multirow{2}{*}{strong scalability} & \multirow{2}{*}{holes} & Dirichlet & \multirow{2}{*}{\ref{tab:metis_partitioning_both}} & \cco & \ccr & \ccg & \ccg & \cco & \cco \\
			& & mixed & & \ccr & \ccr & \cco & \ccg & \ccr & \ccr \\
			\hline
			\multirow{3}{*}{varying $\gamma$} & no holes & Dirichlet & \ref{tab:gamma_no_hole_Dirichlet} & \ccg & \ccg & \ccg & \ccg & \ccg & \ccg \\
			\cline{2-9}
			& holes & Dirichlet & \multirow{2}{*}{\ref{tab:gamma_holes_both}} & \ccr & \cco & \ccg & \ccg & \ccg & \ccg \\
			& holes & mixed & & \ccr & \ccr & \ccr & \ccg & \ccr & \cco \\
			\hline
			\multirow{4}{*}{heterogeneity varying $\varepsilon$} & \multirow{2}{*}{simple} & Dirichlet & \multirow{2}{*}{\ref{tab:epsilon_alternating_both}} & \ccg & \ccg & \ccg & \ccg & \ccg & \ccg \\
			& & mixed & & \ccg & \ccr & \ccg & \ccg & \cco & \cco \\
			\cline{2-9}
			& \multirow{2}{*}{complex} & Dirichlet & \multirow{2}{*}{\ref{tab:epsilon_complex_both}} & \ccg & \ccr & \ccg & \ccg & \ccg & \ccg \\
			& & mixed & & \ccg & \ccr & \ccg & \ccg & \cco & \cco \\
			\hline
			\multirow{4}{*}{heterogeneity varying $\mu$} & \multirow{2}{*}{simple} & Dirichlet & \multirow{2}{*}{\ref{tab:mu_alternating_both}} & \ccg & \ccr & \ccr & \ccg & \ccr & \ccg \\
			& & mixed & & \ccg & \ccr & \ccr & \ccg & \ccr & \ccg \\
			\cline{2-9}
			& \multirow{2}{*}{complex} & Dirichlet & \multirow{2}{*}{\ref{tab:mu_complex_both}} & \cco & \ccr & \ccr & \ccg & \ccr & \ccg \\
			& & mixed & & \cco & \ccr & \ccr & \ccg & \ccr & \ccg \\
		\end{tabular}
	\end{center}
\end{table}

While several variants of our preconditioner are considered, namely an underlying additive Schwarz method with a split near-kernel approach or a global near-kernel approach, both with or without GenEO enrichment, other forms could also be defined and then analysed at the theoretical level, such as those based on an underlying SORAS method or through the use of inexact solvers at the coarse level, as detailed in the preliminary exposition \cite{bootland2020twolevel}.

Presently, we have shown an approach which is essentially a proof of concept, with many directions for further improvement in terms of efficiency. In particular, the coarse space is typically rather sizeable since it contains all gradients of scalar functions, plus a small contribution of local eigenvectors coming from a spectral GenEO coarse space. Nevertheless, the coarse space size is typically five times smaller than the original problem and the gradient part of the coarse matrix retains the sparsity pattern of the underlying finite element discretization. This should enable the design of fast coarse problem solvers. One possible solution is to employ inexact or hierarchical coarse solvers. Such strategies have already received some theoretical backing \cite{bootland2020twolevel} and will be explored in future work.

Our proposed preconditioner must ultimately be tested on possibly highly irregular or heterogeneous problems arising from concrete applications. Its effectiveness should also be investigated for low-frequency Maxwell problems following the framework of \cite{li2012parallel}, where it can be applied within a block preconditioner.

Finally, we note also that the theoretical framework in \Cref{sec:ASwithFSL} is written in an abstract way and is not specific to the positive Maxwell problem. Our approach could thus be applied numerically to any other situation where the discrete problem has a coefficient matrix with a large near-kernel that can be efficiently computed.

\section*{Acknowledgments}

This work was granted access to the HPC resources of IRENE@TGCC under the allocation 2023-067730 granted by GENCI.\\
We also thank Clemens Pechstein (Dassault Systèmes) for his helpful comments on the manuscript and his insights on the AMS preconditioner.

\section*{Data access statement}

Data sharing is not applicable to this article as no datasets were generated or analysed during this study.

\bibliographystyle{siamplain}
\bibliography{refs}
	
\end{document}